\documentclass[3p,11pt]{elsarticle}
\usepackage{url}
\usepackage{amssymb}
\usepackage{amsmath}
\usepackage{stmaryrd}
\usepackage{siunitx}
\usepackage{commath}
\usepackage{subfig}
\usepackage{enumerate}
\usepackage{cleveref}
\makeatletter
\def\ps@pprintTitle{%
 \let\@oddhead\@empty
 \let\@evenhead\@empty
 \def\@oddfoot{}%
 \let\@evenfoot\@oddfoot}

\newcommand{\tnorm}{\@ifstar\@tnorms\@tnorm}
\newcommand{\@tnorms}[1]{%
  \left|\mkern-1.5mu\left|\mkern-1.5mu\left|
   #1
  \right|\mkern-1.5mu\right|\mkern-1.5mu\right|
}
\newcommand{\@tnorm}[2][]{%
  \mathopen{#1|\mkern-1.5mu#1|\mkern-1.5mu#1|}
  #2
  \mathclose{#1|\mkern-1.5mu#1|\mkern-1.5mu#1|}
}
\makeatother

\newcommand{\jump}[1]{\llbracket #1 \rrbracket}

\newtheorem{proposition}{Proposition}
\newdefinition{remark}{Remark}
\newproof{proof}{Proof}
\begin{document}
\begin{frontmatter}
  \title{An exactly mass conserving space-time embedded-hybridized
    discontinuous Galerkin method for the Navier--Stokes equations on
    moving domains}
  \author[TLH]{Tam\'as L. Horv\'ath\corref{cor1}\fnref{label1}}
  \ead{thorvath@oakland.edu}
  \fntext[label1]{\url{https://orcid.org/0000-0001-5294-5362}}
  \address[TLH]{Department of Mathematics and Statistics, Oakland
    University, U.S.A.}
    
  \author[SR]{Sander Rhebergen\fnref{label2}}
  \ead{srheberg@uwaterloo.ca}
  \fntext[label2]{\url{https://orcid.org/0000-0001-6036-0356}}
  \address[SR]{Department of Applied Mathematics, University of
    Waterloo, Canada}  
  \begin{abstract}
    This paper presents a space-time embedded-hybridized discontinuous
    Galerkin (EHDG) method for the Navier--Stokes equations on moving
    domains. This method uses a different hybridization compared to
    the space-time hybridized discontinuous Galerkin (HDG) method we
    presented previously in (Int. J. Numer. Meth. Fluids 89: 519--532,
    2019). In the space-time EHDG method the velocity trace unknown is
    continuous while the pressure trace unknown is discontinuous
    across facets. In the space-time HDG method, all trace unknowns
    are discontinuous across facets. Alternatively, we present also a
    space-time embedded discontinuous Galerkin (EDG) method in which
    all trace unknowns are continuous across facets.  The advantage of
    continuous trace unknowns is that the formulation has fewer global
    degrees-of-freedom for a given mesh than when using discontinuous
    trace unknowns. Nevertheless, the discrete velocity field obtained
    by the space-time EHDG and EDG methods, like the space-time HDG
    method, is exactly divergence-free, even on moving
    domains. However, only the space-time EHDG and HDG methods result
    in divergence-conforming velocity fields. An immediate consequence
    of this is that the space-time EHDG and HDG discretizations of the
    conservative form of the Navier--Stokes equations are energy
    stable. The space-time EDG method, on the other hand, requires a
    skew-symmetric formulation of the momentum advection term to be
    energy-stable. Numerical examples will demonstrate the differences
    in solution obtained by the space-time EHDG, EDG, and HDG methods.
  \end{abstract}
\end{frontmatter}
\section{Introduction}
\label{sec:introduction}

In this paper, we consider the solution of the Navier--Stokes
equations on moving and deforming domains. One of the more popular
approaches to solve partial differential equations on time-dependent
domains is the arbitrary Lagrangian-Eulerian (ALE) class of
methods. In the ALE method, the time-varying domain is mapped to a
fixed reference domain on which all computations are
performed. Although relatively easy to implement, it is known that the
ALE method does not automatically satisfy the Geometric Conservation
Law (GCL) \cite{Lesoinne:1996} which requires that the numerical
method must reproduce exactly a constant solution. Not satisfying the
GCL has consequences for the time-accuracy of the solution
\cite{Guillard:2000} and constraints on the numerical method are
necessary to enforce this property, e.g., \cite{Ivancic:2019,
  Persson:2009}.

A different approach to solving PDEs on deforming domains is by
space-time methods. In space-time finite element methods, the PDE is
discretized simultaneously in space and time by a finite element
method. Typically, discontinuous Galerkin (DG) time-stepping methods
are employed \cite{Eriksson:1985, Jamet:1978} tensorized with some
spatial elemental bases. The space-time finite element method has
successfully been applied to the Navier--Stokes equations, e.g.,
\cite{Johnson:1994, Masud:1997, Tezduyar:1992a, Ndri:2001,
  Ndri:2002}. However, using a continuous spatial basis results in
methods that are not locally conservative and which are generally not
suited for advection dominated flows.

Using a DG basis in both space and time results in the space-time DG
method. This approach was first introduced for compressible flows in
\cite{Vegt:2002} and applied later also to incompressible flows
\cite{Rhebergen:2013b, Vegt:2008}. The space-time DG method is locally
conservative, can be made of arbitrary order in both space and time,
automatically satisfies the GCL, and is well suited for advection
dominated flows. Unfortunately, the space-time DG method is
computationally costly; a $d$-dimensional time-dependent problem is
discretized as a $d+1$-dimensional space-time problem. This results in
a significant increase in the number of degrees-of-freedom compared
to, for example, a $d$-dimensional DG discretization within an ALE
framework.

To address the increase of degrees-of-freedom in a space-time
framework, space-time hybridizable discontinuous Galerkin (HDG)
methods were introduced in \cite{Rhebergen:2012, Rhebergen:2013a}; see
also \cite{Kirk:2019} for an analysis of the space-time HDG method for
the advection-diffusion equation. HDG methods were first introduced
for elliptic problems in \cite{Cockburn:2009a} to reduce the
computational cost of DG methods. This is achieved by introducing
additional facet unknowns in such a way that static-condensation, in
which cell-wise unknowns are eliminated, is trivial. Since its
introduction, the HDG method has successfully been applied to both
compressible \cite{Nguyen:2012} and incompressible
\cite{Cesmelioglu:2017, Fu:2019, Labeur:2012, Lehrenfeld:2016,
  Nguyen:2011, Qiu:2016, Rhebergen:2017, Rhebergen:2018,
  Schroeder:2018} flows.

In \cite{Horvath:2019} we introduced a new space-time HDG method for
the Navier--Stokes equations. The novelty of this new space-time HDG
method is that the discrete velocity is both exactly divergence-free
and divergence-conforming, even on time-dependent domains. A
consequence is that the discretization of the conservative form of the
Navier--Stokes equations is energy-stable. Additionally, the
space-time HDG method is locally mass and momentum conserving.

In the space-time HDG method \cite{Horvath:2019} the discrete velocity
and pressure trace unknowns are discontinuous across facets. As
previously mentioned, these additional facet unknowns are introduced
such that eliminating cell-wise unknowns is trivial. Recently, in
\cite{Rhebergen:2019, Cesmelioglu:2019}, an embedded-hybridized
discontinuous Galerkin (EHDG) method was introduced for incompressible
flows on a fixed domain. In such an approach the discrete velocity
trace is continuous across facets while the pressure trace is
discontinuous. It was shown in \cite{Rhebergen:2019} that the EHDG
method has fewer global degrees-of-freedom than the HDG method, but
retains the advantages of the HDG method, i.e., the discrete velocity
is exactly divergence-free and divergence-conforming.

To reduce the number of globally coupled degrees-of-freedom even
further, we may take both the velocity and pressure trace unknowns to
be continuous across facets. This results in a space-time embedded
discontinuous Galerkin (EDG) method. An EDG method for the
Navier--Stokes equations on fixed domains was introduced in
\cite{Labeur:2012}. Unfortunately, the EDG method results in a
discrete velocity that is not divergence-conforming. The EDG method,
therefore, requires a skew-symmetric formulation of the momentum
advection term to be energy-stable. Furthermore, the EDG method is not
locally mass conserving.

Motivated by the advantages the EHDG method has shown for
incompressible flows on a fixed mesh \cite{Rhebergen:2019}, in this
paper, we introduce a space-time EHDG method for the Navier--Stokes
equations on time-dependent domains. We furthermore generalize the EDG
method \cite{Labeur:2012} to a space-time formulation suitable for
approximating the solution to the Navier--Stokes equations on
moving/deforming domains.

The rest of this paper is organized as follows. In \cref{s:problem} we
introduce the Navier--Stokes problem. We present the space-time EHDG
and EDG methods in \cref{s:discretization} and discuss the properties
of these methods in \cref{s:properties}. We present numerical results
in \cref{s:examples} and draw conclusions in \cref{s:conclusions}.

\section{The Navier--Stokes equations on moving/deforming domains}
\label{s:problem}
Let $\Omega(t) \subset \mathbb{R}^d$ ($d=2,3$) be a domain which
depends continuously on $t \in [0,T]$.  We are interested in the
solution of the Navier--Stokes equations on the space-time domain
$\mathcal{E} := \cbr{(t,x) \,|\, 0<t<T, \, x\in\Omega(t)} \subset
\mathbb{R}^{d+1}$:
\begin{subequations}
  \begin{align}
    \label{eq:navierstokes_a}
    \partial_tu + u\cdot\nabla u - \nu \nabla^2 u + \nabla p
    &= f && \text{in } \mathcal{E}, \\
    \label{eq:navierstokes_b}
    \nabla \cdot u
    &= 0 && \text{in } \mathcal{E}, \\
    \label{eq:navierstokes_c}
    u
    &= 0 && \text{on } \partial\mathcal{E}^D, \\
    \label{eq:navierstokes_d}
    \sbr[0]{n_t + u\cdot n - \max(n_t + u\cdot n, 0)}u + (p \mathbb{I} - \nu \nabla u) n
    &= g && \text{on } \partial\mathcal{E}^N, \\
    \label{eq:navierstokes_e}
    u(0,x)
    &= u_0(x) && \text{in } \Omega(0),
  \end{align}
  \label{eq:navierstokes}
\end{subequations}
where $u : \mathcal{E} \to \mathbb{R}^d$ and
$p : \mathcal{E} \to \mathbb{R}$ are the unknown velocity and
kinematic pressure, respectively, $\nu \in \mathbb{R}^+$ is the
kinematic viscosity, $f : \mathcal{E} \to \mathbb{R}^d$ is a forcing
term, $g : \partial\mathcal{E}^N \to \mathbb{R}^d$ is boundary data,
$u_0 : \Omega(0) \to \mathbb{R}^d$ the initial divergence-free
velocity field, and $\mathbb{I} \in \mathbb{R}^{d\times d}$ is the
identity matrix. Furthermore, the boundary of $\mathcal{E}$ is
partitioned such that
$\partial\mathcal{E} = \partial\mathcal{E}^D \cup
\partial\mathcal{E}^N \cup \Omega(0) \cup \Omega(T)$, where there is
no overlap between any two of the four sets. The space-time outward
unit normal to $\partial\mathcal{E}$ is denoted by
$(n_t, n) \in \mathbb{R}^{d+1}$, with temporal component
$n_t \in \mathbb{R}$ and spatial component $n \in \mathbb{R}^d$.

\section{Discretization}
\label{s:discretization}
%
\subsection{Notation}
\label{ss:notation}
We consider a `slab-by-slab' approach to discretize the space-time
domain $\mathcal{E}$. For this we partition first the time interval
using time levels $0 = t^0 < t^1 < \cdots < t^N = T$. The length of
each time interval $I^n = (t^n, t^{n+1})$ is denoted by
$\Delta t^n = t^{n+1} - t^n$. We then define a space-time slab as
$\mathcal{E}^n = \cbr[0]{(t, x) \in \mathcal{E} \, | \, t \in
  I^n}$. At time $t=t_n$ we denote $\Omega(t)$ by $\Omega^n$. The
boundary of a space-time slab $\mathcal{E}^n$ is then given by
$\Omega^n$, $\Omega^{n+1}$ and
$\partial\mathcal{E}^n := \cbr[0]{ (t, x) \in \partial\mathcal{E} \, |
  \, t \in I^n}$.

We next introduce the triangulation
$\mathcal{T}^n := \cbr{\mathcal{K}}$ of the space-time slab
$\mathcal{E}^n$, which consists of non-overlapping $(d+1)$-dimensional
simplicial space-time elements $\mathcal{K}$. For simplicity we
consider the case of matching meshes at the boundary, $\Omega^n$,
between two space-time slabs $\mathcal{E}^{n-1}$ and
$\mathcal{E}^n$. For non-conforming (hexahedral) space-time meshes,
see for example \cite{Jamet:1978, Ven:2008}. The triangulation of the
space-time domain $\mathcal{E}$ is then denoted by
$\mathcal{T} := \cup_n\mathcal{T}^n$.

On the boundary of a space-time element, $\partial \mathcal{K}$, we
denote the outward unit space-time normal vector by
$\del[0]{n_t^{\mathcal{K}}, n^{\mathcal{K}}} \in \mathbb{R}^{d+1}$,
however, if no confusion arises, we drop the superscript notation. The
boundary of each space-time element $\mathcal{K} \in \mathcal{T}^n$
consists of at most one facet on which $|n_t| = 1$. We denote this
facet by $K^n$ if $n_t = -1$, and by $K^{n+1}$ if $n_t = 1$. The
remaining part of the boundary of $\mathcal{K}$ is denoted by
$\mathcal{Q}^n_{\mathcal{K}} = \partial \mathcal{K} \backslash K^n$ or
$\mathcal{Q}^n_{\mathcal{K}} = \partial \mathcal{K} \backslash
K^{n+1}$.

In a space-time slab $\mathcal{E}^n$, the set of all facets for which
$|n_t| \neq 1$ is denoted by $\mathcal{S}^n$, the union of these
facets is denoted by $\Gamma^n$.  The set of all interior facets is
denoted by $\mathcal{S}_I^n$, while the set of facets that lie on the
boundary of $\mathcal{E}^n$ for which $|n_t| \ne 1$, is denoted by
$\mathcal{S}_B^n$. The set of facets that lie on the Neumann boundary
$\partial\mathcal{E}^N \cap \partial\mathcal{E}^n$ is denoted by
$\mathcal{S}_N^n$.

In the space-time slab $\mathcal{E}^n$ we consider spaces of
discontinuous functions on $\mathcal{T}^n$,
\begin{subequations}
  \label{eq:st_general_func_spc}
  \begin{align}
  \label{eq:st_general_func_spc_a}
    V_h^n &:= \cbr[1]{ v_h \in \sbr[0]{L^2(\mathcal{T})}^d \, | \,
      v_h \in \sbr[0]{P_k(\mathcal{K})}^d\ \forall \mathcal{K} \in \mathcal{T}^n},
    \\
    \label{eq:st_general_func_spc_b}
    Q_h^n &:= \cbr[1]{ q_h \in L^2(\mathcal{T}) \, | \, q_h \in P_{k-1}(\mathcal{K}) \
      \forall \mathcal{K} \in \mathcal{T}^n},
  \end{align}
\end{subequations}
where $P_l(D)$ denotes the space of polynomials of degree $l \ge 0$ on
a domain $D$. We consider also the following finite dimensional
function spaces on $\Gamma^n$,
\begin{subequations}
  \label{eq:st_general_func_spc_facets_hdg}
  \begin{align}
  \label{eq:st_general_func_spc_facets_hdg_a}
    \bar{V}^n_h
    &:= \cbr[1]{ \bar{v}_h \in \sbr[0]{L^2(\mathcal{S})}^d \, | \,
      \bar{v}_h \in \sbr[0]{P_k(S)}^d
      \ \forall S \in \mathcal{S}^n,\ \bar{v}_h = 0 \ \text{on}\ \partial \mathcal{E}^D \cap \partial \mathcal{E}^n},
    \\
    \label{eq:st_general_func_spc_facets_hdg_d}
    \bar{Q}^n_h
    &:= \cbr[1]{ \bar{q}_h \in L^2(\mathcal{S}) \, | \,
      \bar{q}_h \in P_k(S) \ \forall S\in\mathcal{S}^n}.
  \end{align}
\end{subequations}
The space-time HDG, EHDG and EDG methods are characterized by the
choice of finite element function spaces:
\begin{subequations}
  \label{eq:FESpaces}
\begin{align}
  \label{eq:FESpaces_HDG}
  \text{ST-HDG method: } && X_h^{v,n} &:= V_h^n \times \bar{V}_h^n, & X_h^{q,n} &:= Q_h^n \times \bar{Q}_h^n,
  \\
  \label{eq:FESpaces_EHDG}
  \text{ST-EHDG method: } && X_h^{v,n} &:= V_h^n \times (\bar{V}_h^n \cap C(\Gamma^n)), & X_h^{q,n} &:= Q_h^n \times \bar{Q}_h^n,
  \\
  \label{eq:FESpaces_EDG}
  \text{ST-EDG method: } && X_h^{v,n} &:= V_h^n \times (\bar{V}_h^n \cap C(\Gamma^n)), & X_h^{q,n} &:= Q_h^n \times (\bar{Q}_h^n\cap C(\Gamma^n)).
\end{align}
\end{subequations}
Note that the space-time HDG method uses facet spaces that are
discontinuous. The space-time EHDG method uses a continuous facet
velocity space and a discontinuous pressure facet space. The facet
spaces in the space-time EDG method are both continuous. For
notational convenience, we denote function pairs in $X_h^{v,n}$ and
$X_h^{q,n}$ by boldface, e.g.,
$\boldsymbol{v}_h = (v_h, \bar{v}_h) \in X_h^{v,n}$ and
$\boldsymbol{q}_h = (q_h, \bar{q}_h) \in X_h^{q,n}$. Furthermore, we
introduce $X_h^n = X_h^{v,n} \times X_h^{q,n}$.

Let $q_h^+$ and $q_h^-$ denote the traces of a function
$q_h \in Q_h^n$ at an interior facet shared by elements
$\mathcal{K}^+$ and $\mathcal{K}^-$. We introduce the standard `jump'
operator defined as $\jump{q_hn} = q_h^+n^+ + q_h^-n^-$. On a boundary
facet the jump operator is simply defined as $\jump{q_hn} =
q_hn$. Similar expressions hold for vector-valued functions in
$V_h^n$.

  \subsection{The space-time HDG, EHDG and EDG discontinuous Galerkin methods}
\label{ss:h-meh-dg}
In each space-time slab $\mathcal{E}^n$, $n=0,1,\cdots, N-1$, the
discretization of the Navier--Stokes equations on a time-dependent
domain \cref{eq:navierstokes} is given by: find
$(\boldsymbol{u}_h, \boldsymbol{p}_h) \in X^n_h$ such
that
\begin{subequations}
  \begin{align}
    t_h^n(\boldsymbol{u}_h; \boldsymbol{u}_h, \boldsymbol{v}_h) + 
    a_h^n(\boldsymbol{u}_h, \boldsymbol{v}_h)
    + b_h^n(\boldsymbol{p}_h, \boldsymbol{v}_h)
    &= L_h^n(\boldsymbol{v}_h) && \forall \boldsymbol{v}_h \in X_h^{v,n},
    \\
    b_h^n(\boldsymbol{q}_h, \boldsymbol{u}_h) &= 0 && \forall \boldsymbol{q}_h \in X_h^{q,n},
  \end{align}
  \label{eq:discretization}
\end{subequations}
with
\begin{subequations}
  \begin{equation}
    \label{eq:linform_L}
    L_h^n(\boldsymbol{v}_h)
    := 
    \sum_{\mathcal{K}\in\mathcal{T}^n} \int_{\mathcal{K}} f \cdot v_h \dif x\dif t
    - \int_{\partial\mathcal{E}^N\cap I^n} g \cdot \bar{v}_h \dif s 
    + \int_{\Omega^n}u_h^- \cdot v_h \dif x,
  \end{equation}
  where $u_h^- = \lim_{\varepsilon\to 0}u_h(t^n-\varepsilon)$ for
  $n>0$. For the space-time HDG and EHDG methods, when $n=0$, $u_h^-$
  is the projection of the initial condition $u_0$ into
  $V_h^0 \cap H(\text{div})$. For the space-time EDG method, $u_h^-$
  is the projection of the initial condition $u_0$ into $V_h^0$. In
  all cases, the projection is such that $u_h^-$ is exactly
  divergence-free.

  The `space-time Stokes' bilinear forms are given by \cite[Section
  3.2]{Horvath:2019}:
  \begin{align}
    \label{eq:bilinform_a}
    a_h^n(\boldsymbol{u}, \boldsymbol{v})
    :=&
        \sum_{\mathcal{K}\in\mathcal{T}^n} \int_{\mathcal{K}} \nu \nabla u : \nabla v \dif x\dif t
        + \sum_{K\in\mathcal{T}^n} \int_{\mathcal{Q}_{\mathcal{K}}} 
        \frac{\nu \alpha}{h_{\mathcal{K}}}(u - \bar{u}) \cdot (v - \bar{v}) \dif s
    \\
    \nonumber
       & - \sum_{\mathcal{K}\in\mathcal{T}^n} \int_{\mathcal{Q}_{\mathcal{K}}} \nu \sbr[1]{ (u - \bar{u}) \cdot \pd{v}{n} 
         + \pd{u}{n} \cdot (v - \bar{v}) } \dif s,
    \\
    \label{eq:bilinform_b}
    b_h^n(\boldsymbol{p}, \boldsymbol{v})
    := &
         - \sum_{\mathcal{K}\in\mathcal{T}^n} \int_{\mathcal{K}} p \nabla\cdot v \dif x\dif t 
         + \sum_{\mathcal{K}\in\mathcal{T}^n} \int_{\mathcal{Q}_{\mathcal{K}}} (v - \bar{v}) \cdot n \bar{p} \dif s,
  \end{align}
  with $\alpha > 0$ a penalty parameter that needs to be sufficiently
  large to ensure stability. Finally, the space-time convective
  trilinear form is given by
  \begin{equation}
    \label{eq:trilinform_o}
    \begin{split}
      t^n_h(\boldsymbol{w}; & \boldsymbol{u}, \boldsymbol{v})
      \\      
      :=
      & \sum_{\mathcal{K}\in\mathcal{T}^n} \int_{\mathcal{Q}^n_{\mathcal{K}}}
      \del{n_t + w\cdot n}\del{u + \lambda\del{\bar{u} - u}} \cdot (v - \bar{v}) \dif s
      - \sum_{\mathcal{K}\in\mathcal{T}^n} \int_{\mathcal{Q}^n_{\mathcal{K}}} \tfrac{1}{2} (w \cdot n) \del{ u \cdot v - \bar{u} \cdot \bar{v} } \dif s
      \\
      &- \sum_{\mathcal{K}\in\mathcal{T}^n} \int_{\mathcal{K}} \del{u\partial_tv + u\otimes w : \nabla v} \dif x\dif t
      +\sum_{\mathcal{K}\in\mathcal{T}^n} \int_{\mathcal{K}} \tfrac{1}{2} \del{u\otimes w : \nabla v + v\otimes w : \nabla u} \dif x\dif t      
      \\
      &+ \int_{\partial\mathcal{E}^N \cap I_n}\max\del{n_t + \bar{w}\cdot n, 0}\bar{u}\cdot \bar{v} \dif s      
      - \int_{\partial\mathcal{E}^N \cap I_n} \tfrac{1}{2} \del{\bar{w}\cdot n}\bar{u}\cdot \bar{v} \dif s
      + \sum_{\mathcal{K}\in\mathcal{T}^n} \int_{K^{n+1}}u \cdot v \dif x.
    \end{split}
  \end{equation}
\end{subequations}
The space-time convective trilinear form \cref{eq:trilinform_o} is an
extension of the convective trilinear form we proposed in
\cite{Horvath:2019} for space-time HDG methods. Indeed, it can be
shown that if $w$ in \cref{eq:trilinform_o} is divergence-conforming
and point-wise divergence-free, the convective trilinear form reduces
to:
\begin{multline}
  \label{eq:trilinform_hr}
  t_h^n(\boldsymbol{w}; \boldsymbol{u}, \boldsymbol{v}) := 
  \sum_{\mathcal{K}\in\mathcal{T}^n} \int_{K^{n+1}}u \cdot v \dif x
  + \sum_{\mathcal{K}\in\mathcal{T}^n} \int_{\mathcal{Q}^n_{\mathcal{K}}}
  \del{n_t + w\cdot n}\del{u + \lambda\del{\bar{u} - u}} \cdot (v - \bar{v}) \dif s
  \\
  + \int_{\partial\mathcal{E}^N \cap I_n}\max\del{n_t + \bar{w}\cdot n, 0}\bar{u}\cdot \bar{v} \dif s
  - \sum_{\mathcal{K}\in\mathcal{T}^n} \int_{\mathcal{K}} \del{u\partial_tv + u\otimes w : \nabla v} \dif x\dif t,
\end{multline}
which is the discretization of the momentum advection term in
conservative form. Provided that $u_h$ is divergence-conforming and
point-wise divergence-free, we showed \cite[Section 4]{Horvath:2019}
that the space-time discretization \cref{eq:discretization}, with the
trilinear form given by \cref{eq:trilinform_hr}, is energy-stable.

The difference between \cref{eq:trilinform_o} and
\cref{eq:trilinform_hr}, for a $w$ in \cref{eq:trilinform_o} that is
not both divergence-conforming and point-wise divergence-free, is the
following consistent `energy-stabilization' term:
\begin{multline}
  \label{eq:energy-stab}
  e_h^n(\boldsymbol{w}; \boldsymbol{u}, \boldsymbol{v}) :=
  \sum_{\mathcal{K}\in\mathcal{T}^n} \int_{\mathcal{K}} \tfrac{1}{2} \del{u\otimes w : \nabla v + v\otimes w : \nabla u} \dif x\dif t
  - \sum_{\mathcal{K}\in\mathcal{T}^n} \int_{\mathcal{Q}^n_{\mathcal{K}}} \tfrac{1}{2} (w \cdot n) \del{ u \cdot v - \bar{u} \cdot \bar{v} } \dif s
  \\
  - \int_{\partial\mathcal{E}^N \cap I_n} \tfrac{1}{2} \del{\bar{w}\cdot n}\bar{u}\cdot \bar{v} \dif s.
\end{multline}
This term results in a skew-symmetric discretization of the momentum
advection term. This is necessary to prove energy-stability of
\cref{eq:discretization} in the event that $u_h \in V_h$ is not both
divergence-conforming and point-wise divergence-free, as we will show
in \cref{s:properties}.

\subsection{The Oseen problem}
\label{ss:picard}

The space-time discretizations of the Navier--Stokes equations
\cref{eq:discretization} results in a system of nonlinear algebraic
equations in each space-time slab. We use Picard iteration to solve
these systems of equations; given
$(\boldsymbol{u}_h^k, \boldsymbol{p}_h^k) \in X_h^n$ we seek
$(\boldsymbol{u}_h^{k+1}, \boldsymbol{p}_h^{k+1}) \in X_h^n$ such that
\begin{subequations}
  \begin{align}
    t_h^n(\boldsymbol{u}_h^k; \boldsymbol{u}_h^{k+1}, \boldsymbol{v}_h) + 
    a_h^n(\boldsymbol{u}_h^{k+1}, \boldsymbol{v}_h)
    + b_h^n(\boldsymbol{p}_h^{k+1}, \boldsymbol{v}_h)
    &= L_h^n(\boldsymbol{v}_h) && \forall \boldsymbol{v}_h \in X_h^{v,n},
    \\
    b_h^n(\boldsymbol{q}_h, \boldsymbol{u}_h^{k+1}) &= 0 && \forall \boldsymbol{q}_h \in X_h^{q,n},
  \end{align}
  \label{eq:Picard}
\end{subequations}
for $k=0,1,2,\hdots$. Once a convergence criterion has been met we set
$(\boldsymbol{u}_h, \boldsymbol{p}_h) = (\boldsymbol{u}_h^{k+1},
\boldsymbol{p}_h^{k+1})$. We note that the linearized discretization
\cref{eq:Picard} is a space-time discretization of the Oseen
equations.

\section{Properties of the space-time HDG, EHDG and EDG discretizations}
\label{s:properties}

In this section we discuss properties of the space-time discretizations
of the Navier--Stokes equations \cref{eq:discretization}. We start by
showing that only the space-time HDG and EHDG discretizations are
locally mass conserving.

\begin{proposition}[Local mass conservation]
  \label{prop:local_mass_conservation}
  The space-time HDG and EHDG methods defined in
  \cref{eq:discretization} are locally mass conserving.
\end{proposition}
\begin{proof}
  The proof is similar to \cite[Prop. 1]{Rhebergen:2018}. We note
  first that setting $\boldsymbol{v}_h = 0$, $\bar{q}_h = 0$ and
  $q_h = \nabla \cdot u_h$ in \cref{eq:discretization} immediately
  results in $\nabla \cdot u_h = 0$ for all $(t,x) \in \mathcal{K}$,
  $\forall \mathcal{K} \in \mathcal{T}^n$, i.e., the approximate
  velocity is exactly divergence-free.

  We furthermore note that setting $(v_h, \bar{v}_h, q_h) = (0, 0, 0)$
  and
  $\bar{q}_h = \jump{\del[0]{u_h^k - \bar{u}_h^k}\cdot n} \in
  \bar{Q}_h^n$ in \cref{eq:discretization} results in
  $\jump{u_h \cdot n} = 0$ for all
  $(t,x) \in \mathcal{S}, \forall \mathcal{S} \in \mathcal{S}_I^n$ and
  $u_h \cdot n = \bar{u}_h \cdot n$ for all
  $(t,x) \in \mathcal{S}, \forall \mathcal{S} \in \mathcal{S}_B^n$,
  i.e., the approximate velocity is divergence-conforming.

  Since the approximate velocity is both divergence-free and
  divergence-conforming, the result follows. \qed
\end{proof}

\begin{remark}
  \label{rem:edg-div-free}
  The approximate velocity obtained by the space-time EDG method is
  divergence-free but not divergence-conforming. This is because
  $\jump{\del[0]{u_h^k - \bar{u}_h^k}\cdot n} \notin \bar{Q}_h^n \cap
  C(\Gamma^n)$. As a result, the space-time EDG method is not locally
  mass conserving.
\end{remark}

We next show energy-stability of all three space-time methods.

\begin{proposition}[energy-stability]
  \label{prop:sthdg_edghdg_energy}
  The space-time methods defined in \cref{eq:discretization} are
  energy-stable.
\end{proposition}
\begin{proof}
  The space-time HDG and EHDG methods both result in an approximate
  velocity that is divergence-free and divergence-conforming, see
  \Cref{prop:local_mass_conservation}. As a result, the trilinear form
  $t_h^n(\cdot; \cdot, \cdot)$ in \cref{eq:discretization} is
  identical to the trilinear form given in
  \cref{eq:trilinform_hr}. The proof of energy-stability now follows
  identical steps as in \cite[Section 4]{Horvath:2019} and is
  therefore omitted.

  We prove now energy-stability of the space-time EDG method. Consider
  the first space-time slab, and assume that we have obtained
  $(\boldsymbol{u}_h^k, \boldsymbol{p}_h^k)$ from the $k$th Picard
  iteration~\cref{eq:Picard}. To simplify notation, we write
  $\boldsymbol{w}_h = \boldsymbol{u}_h^k$ and
  $\boldsymbol{u}_h = \boldsymbol{u}_h^{k+1}$. For homogeneous
  boundary conditions and setting $f=0$ and
  $(\boldsymbol{v}_h, \boldsymbol{q}_h) = (\boldsymbol{u}_h,
  \boldsymbol{p}_h^{k+1})$ in \cref{eq:Picard},
  \begin{equation}
    \label{eq:alternative_wform}
    t_h^0(\boldsymbol{w}_h, \boldsymbol{u}_h, \boldsymbol{u}_h) + 
    a_h^0(\boldsymbol{u}_h, \boldsymbol{u}_h) =
    \int_{\Omega_0}u_h^- \cdot u_h \dif s.
  \end{equation}
  For $\alpha > 0$ large enough we know that
  $a_h^0(\boldsymbol{u}_h, \boldsymbol{u}_h) \ge 0$
  \cite{Rhebergen:2017, Wells:2011}. This implies that for
  sufficiently large $\alpha > 0$
  \begin{equation}
    \label{eq:alternative_wform_pm}
    t_h^0(\boldsymbol{w}_h, \boldsymbol{u}_h, \boldsymbol{u}_h)
    - \int_{\Omega_0}u_h^- \cdot u_h \dif s
    \le 0.
  \end{equation}
  Using that
  $u_h \cdot (u_h - \bar{u}_h) - \tfrac{1}{2} (\envert{u_h}^2 -
  \envert{\bar{u}_h}^2) = \tfrac{1}{2} \envert{u_h - \bar{u}_h}^2$ the
  trilinear form may be written as
  \begin{equation}
    \label{eq:rearrange2}
    \begin{split}
      t_h^0(\boldsymbol{w}_h, \boldsymbol{u}_h, \boldsymbol{u}_h)
      &:=
      \sum_{\mathcal{K}\in\mathcal{T}^0} \int_{K^{1}}\envert{u_h}^2 \dif s
      + \sum_{\mathcal{K}\in\mathcal{T}^0} \int_{\mathcal{Q}^0_{\mathcal{K}}} 
      n_t u_h \cdot (u_h - \bar{u}_h) \dif s
      \\
      &- \sum_{\mathcal{K}\in\mathcal{T}^0} \int_{\mathcal{Q}^0_{\mathcal{K}}} 
      \lambda \del{n_t + w_h\cdot n}\envert{u_h - \bar{u}_h}^2 \dif s
      + \sum_{\mathcal{K}\in\mathcal{T}^0} \int_{\mathcal{Q}^0_{\mathcal{K}}} 
      \tfrac{1}{2}(w_h \cdot n) \envert{u_h - \bar{u}_h}^2 \dif s
      \\
      &+ \int_{\partial\mathcal{E}^N \cap I_0}\max(n_t + \bar{w}_h\cdot n, 0)\envert{\bar{u}_h}^2 \dif s
      - \int_{\partial\mathcal{E}^N \cap I_0} \tfrac{1}{2} \del{\bar{w}_h\cdot n}\envert{\bar{u}_h}^2 \dif s
      \\
      &- \sum_{\mathcal{K}\in\mathcal{T}^0} \int_{\mathcal{K}} \tfrac{1}{2}\partial_t\envert{u_h}^2 \dif x\dif t.      
    \end{split}
  \end{equation}
  Adding and subtracting
  $\int_{\mathcal{Q}^0_{\mathcal{K}}}\tfrac{1}{2} n_t \envert{u_h -
    \bar{u}_h}^2 \dif s$ and
  $\int_{\partial\mathcal{E}^N \cap I_0} \tfrac{1}{2}
  n_t\envert{\bar{u}_h}^2 \dif s$ results in
  \begin{equation}
    \label{eq:rearrange3}
    \begin{split}
      t_h^0(\boldsymbol{w}_h, \boldsymbol{u}_h, \boldsymbol{u}_h)
      &:=
      \sum_{\mathcal{K}\in\mathcal{T}^0} \int_{K^{1}}\envert{u_h}^2 \dif s
      + \sum_{\mathcal{K}\in\mathcal{T}^0} \int_{\mathcal{Q}^0_{\mathcal{K}}} 
      n_t u_h \cdot (u_h - \bar{u}_h) \dif s
      \\
      &- \sum_{\mathcal{K}\in\mathcal{T}^0} \int_{\mathcal{Q}^0_{\mathcal{K}}} 
      \lambda \del{n_t + w_h\cdot n}\envert{u_h - \bar{u}_h}^2 \dif s
      + \sum_{\mathcal{K}\in\mathcal{T}^0} \int_{\mathcal{Q}^0_{\mathcal{K}}} 
      \tfrac{1}{2}\del{n_t + w_h \cdot n} \envert{u_h - \bar{u}_h}^2 \dif s
      \\
      &+ \int_{\partial\mathcal{E}^N \cap I_0}\max(n_t + \bar{w}_h\cdot n, 0)\envert{\bar{u}_h}^2 \dif s
      - \int_{\partial\mathcal{E}^N \cap I_0} \tfrac{1}{2} \del{n_t + \bar{w}_h\cdot n}\envert{\bar{u}_h}^2 \dif s
      \\
      &-\int_{\mathcal{Q}^0_{\mathcal{K}}}\tfrac{1}{2} n_t \envert{u_h -
        \bar{u}_h}^2 \dif s
      +\int_{\partial\mathcal{E}^N \cap I_0} \tfrac{1}{2} n_t\envert{\bar{u}_h}^2 \dif s
      - \sum_{\mathcal{K}\in\mathcal{T}^0} \int_{\mathcal{K}} \tfrac{1}{2}\partial_t\envert{u_h}^2 \dif x\dif t.      
    \end{split}
  \end{equation}
  We now simplify this expression. We first note that the third and
  fourth terms on the right hand side of \cref{eq:rearrange3} may be
  combined to
  \begin{equation}
    \sum_{\mathcal{K}\in\mathcal{T}^0} \int_{\mathcal{Q}^0_{\mathcal{K}}} 
    \del{\tfrac{1}{2}-\lambda}\del{n_t + w_h \cdot n} \envert{u_h - \bar{u}_h}^2 \dif s  
    =
    \sum_{\mathcal{K}\in\mathcal{T}^0} \int_{\mathcal{Q}^0_{\mathcal{K}}} 
    \tfrac{1}{2}\envert{n_t + w_h \cdot n} \envert{u_h - \bar{u}_h}^2 \dif s.
  \end{equation}
  Note also that the fifth and sixth terms on the right hand side
  of \cref{eq:rearrange3} may be combined to
  \begin{equation}
    \int_{\partial\mathcal{E}^N \cap I_0}\del{\max(n_t + \bar{w}_h\cdot n, 0)-\tfrac{1}{2}\del{n_t + \bar{w}_h\cdot n}}\envert{\bar{u}_h}^2 \dif s
    = 
    \int_{\partial\mathcal{E}^N \cap I_0}\tfrac{1}{2}
    \envert{n_t + \bar{w}_h\cdot n}\envert{\bar{u}_h}^2 \dif s.
  \end{equation}
  We therefore write \cref{eq:rearrange3} as
  \begin{equation}
    \label{eq:rearrange4}
    \begin{split}
      t_h^0(\boldsymbol{w}_h, \boldsymbol{u}_h, \boldsymbol{u}_h)
      :=&
      \sum_{\mathcal{K}\in\mathcal{T}^0} \int_{K^{1}}\envert{u_h}^2 \dif s
      + \sum_{\mathcal{K}\in\mathcal{T}^0} \int_{\mathcal{Q}^0_{\mathcal{K}}} 
      n_t u_h \cdot (u_h - \bar{u}_h) \dif s
      \\
      &+ \sum_{\mathcal{K}\in\mathcal{T}^0} \int_{\mathcal{Q}^0_{\mathcal{K}}} 
      \tfrac{1}{2}\envert{n_t + w_h \cdot n} \envert{u_h - \bar{u}_h}^2 \dif s
      + \int_{\partial\mathcal{E}^N \cap I_0}\tfrac{1}{2}
      \envert{n_t + \bar{w}_h\cdot n}\envert{\bar{u}_h}^2 \dif s
      \\
      &-\int_{\mathcal{Q}^0_{\mathcal{K}}}\tfrac{1}{2} n_t \envert{u_h -
        \bar{u}_h}^2 \dif s
      +\int_{\partial\mathcal{E}^N \cap I_0} \tfrac{1}{2} n_t\envert{\bar{u}_h}^2 \dif s
      - \sum_{\mathcal{K}\in\mathcal{T}^0} \int_{\mathcal{K}} \tfrac{1}{2}\partial_t\envert{u_h}^2 \dif x\dif t.
    \end{split}
  \end{equation}
  Combining \cref{eq:rearrange4} with \cref{eq:alternative_wform_pm},
  \begin{multline}
    \label{eq:rearrange5}
    \sum_{\mathcal{K}\in\mathcal{T}^0} \int_{K^{1}}\envert{u_h}^2 \dif s
    + \sum_{\mathcal{K}\in\mathcal{T}^0} \int_{\mathcal{Q}^0_{\mathcal{K}}} 
    n_t u_h \cdot (u_h - \bar{u}_h) \dif s
    -\int_{\mathcal{Q}^0_{\mathcal{K}}}\tfrac{1}{2} n_t \envert{u_h-\bar{u}_h}^2 \dif s
    \\
    +\int_{\partial\mathcal{E}^N \cap I_0} \tfrac{1}{2} n_t\envert{\bar{u}_h}^2 \dif s
    - \sum_{\mathcal{K}\in\mathcal{T}^0} \int_{\mathcal{K}} \tfrac{1}{2}\partial_t\envert{u_h}^2 \dif x\dif t
    - \int_{\Omega_0}u_h^- \cdot u_h \dif s \le 0.
  \end{multline}
  Using the following identities
  \begin{align*}
    \sum_{\mathcal{K}\in\mathcal{T}^0} \int_{\mathcal{K}} \partial_t\envert{u_h}^2 \dif x\dif t
    =&
       \sum_{\mathcal{K}\in\mathcal{T}^0} \int_{K^1} \envert{u_h}^2 \dif x
       - \sum_{\mathcal{K}\in\mathcal{T}^0} \int_{K^0} \envert{u_h}^2 \dif x
       + \sum_{\mathcal{K}\in\mathcal{T}^0} \int_{\mathcal{Q}^0_{\mathcal{K}}} \envert{u_h}^2 n_t \dif s,\\
    \tfrac{1}{2}n_t\envert{\bar{u}_h}^2 
    =&
       \tfrac{1}{2}|u_h|^2 n_t - n_t |u_h|^2 + n_tu_h\cdot\bar{u}_h 
       +\tfrac{1}{2}n_t |u_h-\bar{u}_h|^2,
  \end{align*}
  and the single-valuedness of $\bar{u}_h$ on facets so that
  \begin{equation*}
    \sum_{\mathcal{K}\in\mathcal{T}^0} \int_{\mathcal{Q}^0_{\mathcal{K}}}
    \tfrac{1}{2}n_t\envert{\bar{u}_h}^2 \dif s = 
    \int_{\partial\mathcal{E}^N \cap I_0}\tfrac{1}{2}n_t\envert{\bar{u}_h}^2 \dif s,
  \end{equation*}
  we simplify \cref{eq:rearrange5} to
  \begin{equation}
    \label{eq:rearrange7}
    \sum_{\mathcal{K}\in\mathcal{T}^0} \int_{K^{1}}\tfrac{1}{2}\envert[0]{u_h}^2 \dif x
    + \sum_{\mathcal{K}\in\mathcal{T}^0} \int_{K^0} \tfrac{1}{2}\envert[0]{u_h}^2 \dif x
    - \int_{\Omega_0}u_h^- \cdot u_h \dif x \le 0.
  \end{equation}
  Since
  \begin{equation}
    \label{eq:usefulequality}
    - \int_{\Omega_0} u_h^- \cdot u_h \dif x = 
    \frac{1}{2} \int_{\Omega_0} \envert[0]{u_h - u_h^-}^2 \dif x
    - \frac{1}{2} \int_{\Omega_0} \envert[0]{u_h}^2 \dif x
    - \frac{1}{2} \int_{\Omega_0} \envert[0]{u_h^-}^2 \dif x,
  \end{equation}
  \cref{eq:rearrange7} implies
  \begin{equation}
    \label{eq:senergystab2}
    \int_{\Omega_1} \envert{u_h}^2 \dif x
    \le \int_{\Omega_0} \envert[0]{u_h^-}^2 \dif x.
  \end{equation}
  energy-stability is now proven for all $n > 0$ by using $u_h$ from
  space-time slab $\mathcal{E}_{n-1}$ as initial condition for the
  discretization in space-time slab $\mathcal{E}_n$. \qed
\end{proof}

\begin{proposition}[Local momentum conservation]
  \label{prop:local_mom_conservation}
  The space-time HDG and EHDG methods defined in
  \cref{eq:discretization} locally conserve momentum.
\end{proposition}
\begin{proof}
  The proof is similar to that of \cite[Prop. 4.3]{Labeur:2012}. In
  \cref{eq:discretization}, using the conservative form of the
  trilinear form \cref{eq:trilinform_hr}, set
  $\boldsymbol{v}_h=(e_j, 0)$ on $\mathcal{K}$, where $e_j$ is a
  canonical unit basis vector, $\boldsymbol{v}_h=0$ on
  $\mathcal{T}^n \backslash \mathcal{K}$, and $\boldsymbol{q}_h=0$ on
  $\mathcal{T}^n$:
  \begin{equation}
    \int_{K^{n+1}} u_h \cdot e_j \dif x - \int_{K^n} u_h^- \cdot e_j \dif x
    =
    \int_{\mathcal{K}} f \cdot e_j \dif x \dif t
    -
    \int_{\mathcal{Q}_{\mathcal{K}}} \hat{\sigma}_h \cdot e_j \dif s,
  \end{equation}
  where $\hat{\sigma}_h$ is the `numerical' momentum flux in the
  space-time normal direction on cell boundaries given by
  \begin{equation}
    \hat{\sigma}_h = (n_t + u_h \cdot n)(u_h + \lambda(\bar{u}_h - u_h))
    + \sbr[1]{\bar{p}_h\mathbb{I} - \nu \nabla u_h - \tfrac{\nu \alpha}{h}(\bar{u}_h - u_h)\otimes n} n.
  \end{equation}
  It follows that 
  \begin{equation}
    \int_{K^{n+1}} u_h \dif x - \int_{K^n} u_h^- \dif x
    =
    \int_{\mathcal{K}} f \dif x \dif t
    -
    \int_{\mathcal{Q}_{\mathcal{K}}} \hat{\sigma}_h \dif s \qquad \forall \mathcal{K} \in \mathcal{T},
  \end{equation}
  and the result follows. \qed
\end{proof}

As remarked in \cite{Labeur:2012}, local momentum conservation is in
terms of the numerical flux $\hat{\sigma}_h$. This is a typical
feature of (space-time) discontinuous Galerkin methods. It should be
noted, however, that the normal component of the `numerical' momentum
flux is continuous across facets only in the case of the space-time
HDG method.

\begin{remark}
  A discretization of the conservative form of the momentum equation
  is required to conserve momentum. For the space-time EDG method to
  be energy-stable, it requires a skew-symmetric formulation of the
  momentum advection term, see \Cref{prop:sthdg_edghdg_energy}. The
  space-time EDG method, therefore, cannot be momentum conserving.
\end{remark}

We end this section discussing the number of globally coupled
degrees-of-freedom. Due to static condensation, the number of globally
coupled degrees-of-freedom are determined only by the facet velocity
and facet pressure function spaces.

The facet pressure space for the space-time HDG and EHDG methods are
identical. The difference between these two methods lies therefore in
the facet velocity approximation; in the space-time HDG method the
facet velocity is discontinuous across facets while it is continuous
in the space-time EHDG method. Using a continuous facet velocity
significantly decreases the number of globally coupled
degrees-of-freedom compared to using a discontinuous facet velocity,
especially in higher dimensions.

The facet velocity space for the space-time EHDG and EDG methods are
identical, but the facet pressure space differs. In the space-time
EHDG method the facet pressure approximation is discontinuous across
facets. It is continuous across facets in the space-time EDG
method. However, since the facet pressure is a scalar, the reduction
in the number of globally coupled degrees-of-freedom when replacing a
discontinuous facet pressure space by a continuous facet pressure
space is less significant than in the case of the facet velocity
space.


We summarize the properties of the space-time HDG, EHDG, and EDG
methods in \cref{tab:properties}.

\begin{table}[tbp]
  \centering
  \begin{tabular}{c|ccc}
    \hline
    & ST-HDG & ST-EHDG & ST-EDG \\
    \hline
    divergence-free velocity       & $\checkmark$ & $\checkmark$ & $\checkmark$ \\ \hline
    divergence-conforming velocity & $\checkmark$ & $\checkmark$ & $\times$     \\ \hline
    energy-stable                  & $\checkmark$ & $\checkmark$ & $\checkmark$ \\ \hline
    locally momentum conserving    & $\checkmark$ & $\checkmark$ & $\times$     \\ \hline
    number of degrees-of-freedom   & largest      & significantly less & slightly less \\
                                   &              & than ST-HDG        & than ST-EHDG  \\ \hline
  \end{tabular}
  \caption{Comparison of the properties of the space-time HDG, EHDG and EDG methods.}
  \label{tab:properties}
\end{table}

\section{Numerical examples}
\label{s:examples}

All simulations in this section were carried out using the Modular
Finite Element Method (MFEM) library \cite{mfem-library}. Furthermore,
as is common for interior penalty DG methods \cite{Labeur:2012,
  Rhebergen:2018}, we choose a penalty parameter of the form
$\alpha = ck^2$, where $k$ is the order of the polynomial
approximation and $c$ a constant. We take $c=6$ in all our
simulations.

The non-linear problem \cref{eq:discretization} in each time-slab
$\mathcal{E}^n$ ($n=0,\cdots, N-1$) is solved by Picard iteration
\cref{eq:Picard}. We set $u_h^0 = 0, p_h^0 = 0$ and use as stopping
criterion
\begin{equation}
  \label{eq:stopping_criterion}
  \max \cbr{ \frac{\norm[0]{u_h^k - u_h^{k-1}}_{\infty}}{\norm[0]{u_h^k - u_h^0}_{\infty}},
    \frac{\norm[0]{p_h^k - p_h^{k-1}}_{\infty}}{\norm[0]{p_h^k - p_h^0}_{\infty}} }
  < \text{TOL},
\end{equation}
where $\norm[0]{\cdot}_\infty$ is the discrete $l^\infty$-norm, and
$\text{TOL}$ the desired tolerance.

Let
$U\in \mathbb{R}^{\dim \boldsymbol{V}^n_h}, P\in \mathbb{R}^{\dim
  Q^n_h}, \bar{U}\in \mathbb{R}^{\dim \bar{\boldsymbol{V}}^n_h},
\bar{P}\in \mathbb{R}^{\dim \bar{Q}^n_h}$ be the vectors of
coefficients of
$\boldsymbol{u}_h, p_h, \bar{\boldsymbol{u}}_h, \bar{p}_h$ with
respect to the basis of the corresponding vector spaces. Then
$W^T = [U^T\ P^T]$ is the vector of all element degrees-of-freedom and
$\bar{W}^T = [\bar{U}^T\ \bar{P}^T]$ is the vector of all facet
degrees-of-freedom. At each Picard iteration~\cref{eq:Picard} the
linear system can be written in the following block-matrix form:
\begin{equation}
  \label{eq:before_SC}
  \begin{bmatrix}
    A & B \\ C & D
  \end{bmatrix}
  \begin{bmatrix}
    W \\ \bar{W}
  \end{bmatrix}
  =
  \begin{bmatrix}
    F \\ \bar{F}
  \end{bmatrix}.
\end{equation}
As with all other hybridizable discontinuous Galerkin methods, $A$ has
a block-diagonal structure. It is therefore cheap to eliminate $W$
from~\cref{eq:before_SC} to obtain the reduced linear system
\begin{equation}
  \label{eq:SC}
(-CA^{-1}B + D) \bar{W} = \bar{F} - CA^{-1}F.
\end{equation}
We use the direct solver of MUMPS~\cite{MUMPS:1,MUMPS:2} through
PETSc~\cite{petsc-web-page,petsc-user-ref,petsc-efficient} to solve
this system of linear equations. Given $\bar{W}$ we can then compute
$W$ cell-wise according to $W = A^{-1}(F - B\bar{W})$.

\subsection{Convergence rates}
\label{ss:convrates}
In this first test case, we compute the rates of convergence of the
space-time HDG, EHDG and EDG methods applied to the Navier--Stokes equations on a
time-dependent domain. Introducing first a uniform triangular mesh for
the unit square, the mesh vertices $(x_1, x_2)$ for the deforming
domain $\Omega(t)$ are obtained at any time $t\in[0,1]$ by the
following relation
\begin{equation*}
  x_i = x_i^0 + 0.05 (1 - x_i^0)\sin( 2\pi (\tfrac{1}{2} - x_i^* + t) ) 
  \qquad  i = 1,2,  
\end{equation*}
where $(x_1^0, x_2^0) \in [0,1]^2$ are the vertices of the uniform
mesh and $(x_1^*, x_2^*) = (x_2^0, x_1^0)$. 

Let
$\partial\mathcal{E}^N := \cbr{(t,x_1,x_2)\in\partial\mathcal{E} :
  x_1=1}$ and
$\partial\mathcal{E}^D = \partial \mathcal{E} \setminus
(\partial\mathcal{E}^N \cup \Omega(0) \cup \Omega(1))$. The boundary
conditions and source term $f$ in~\cref{eq:navierstokes_a} are chosen
such that the exact solution is given by
\begin{equation*}
  u = 
  \begin{bmatrix}
    2 + \sin(2\pi (x_1 - t)) \sin(2\pi (x_2 - t))\\
    2 + \cos(2\pi (x_1 - t)) \cos(2\pi (x_2 - t))
  \end{bmatrix},
  \qquad
  p = \sin(2\pi (x_1 - t)) \cos(2\pi (x_2 - t)).
\end{equation*}
The deforming mesh and pressure solution at three different points
in time are shown in \cref{fig:mesh_movement}.
\begin{figure}[tbp]
  \begin{center}
    \includegraphics[width=.32\linewidth]{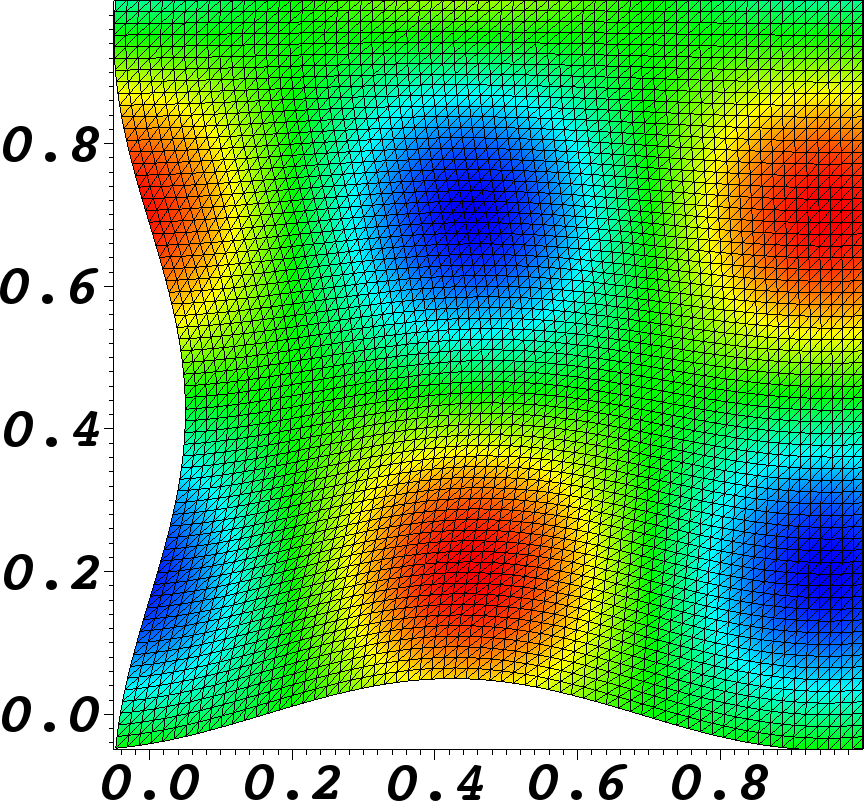} 
    \includegraphics[width=.32\linewidth]{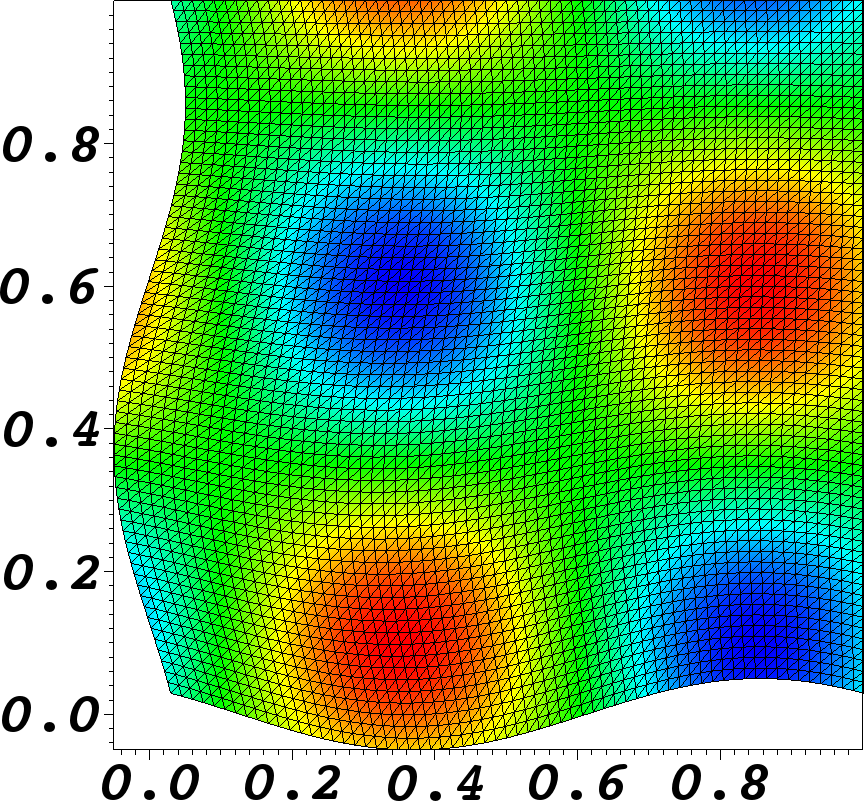}
    \includegraphics[width=.32\linewidth]{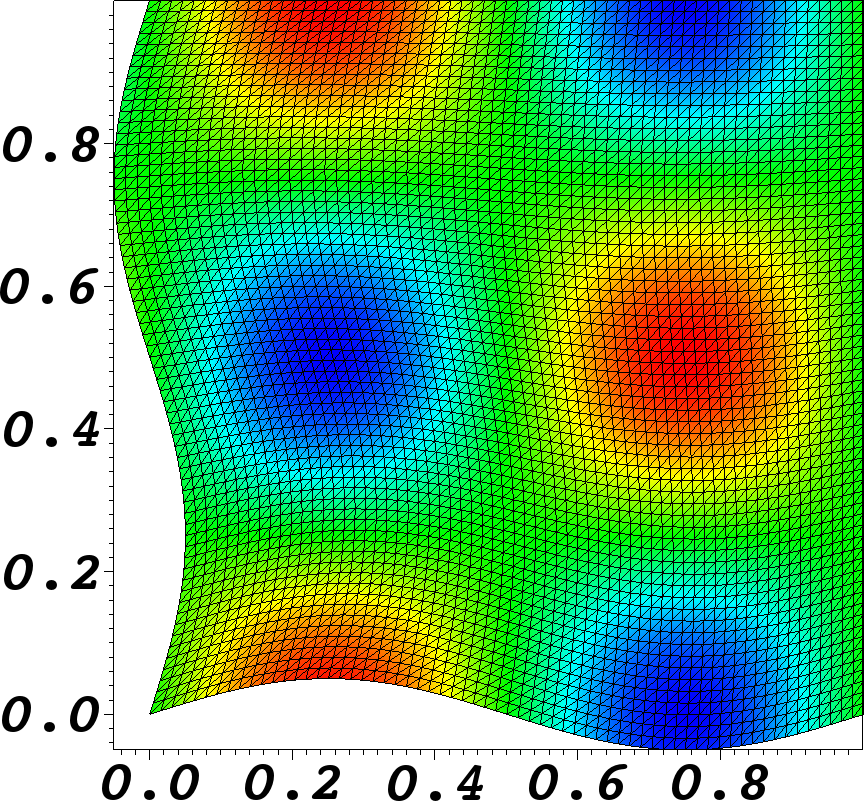}
    \caption{The mesh and pressure solution at different points in time for the test case
      described in \cref{ss:convrates}. From left to right the mesh and pressure solution at
      $t = 0.2,\ 0.6,\ 1.0$.}
    \label{fig:mesh_movement}
  \end{center}
\end{figure}

We consider the rates of convergence for polynomial degrees $k=2$ and
$k=3$ and on a succession of refined space-time meshes. The coarsest
space-time mesh consists of $6 \cdot 8^2$ tetrahedra per space-time
slab with $\Delta t = 0.05$. For the Picard
iteration \cref{eq:stopping_criterion} we set $\text{TOL} = 10^{-12}$.

The rates of convergence over the entire space-time domain
$\mathcal{E}$, with $\nu = 10^{-7}$, are shown in
\cref{fig:error_vs_dof}. We observe that all space-time methods
converge optimally, i.e., the velocity error is of order
$\mathcal{O}(h^{k+1})$ and the pressure error is of order
$\mathcal{O}(h^k)$. We observe that the space-time EDG and EHDG
methods give smaller errors than the space-time HDG method for the
same number of globally coupled degrees-of-freedom.

\begin{figure}[tbp]
  \begin{center}
    \includegraphics[width=0.49\linewidth]{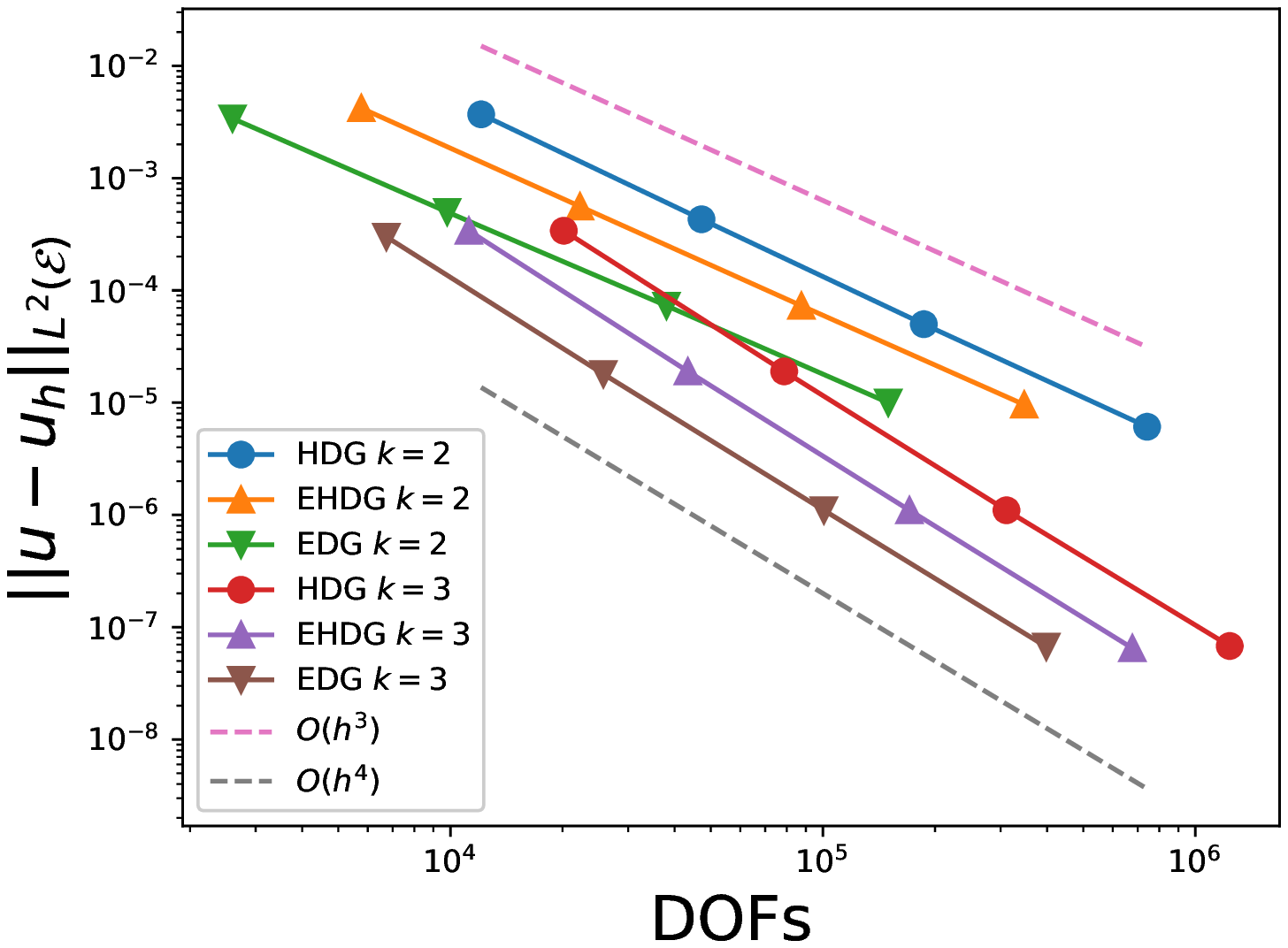}
    \includegraphics[width=0.49\linewidth]{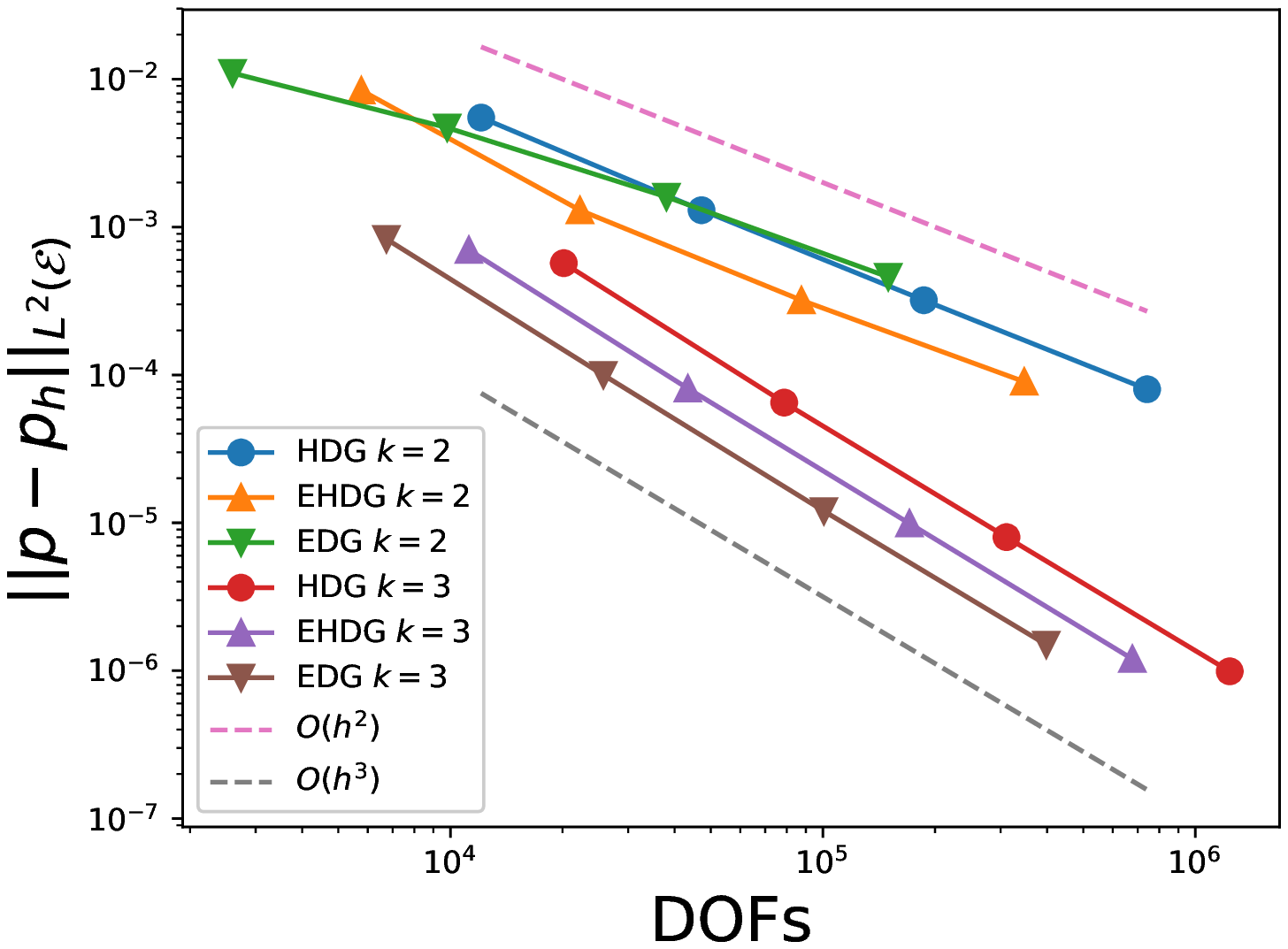}
    \caption{$L^2$-norm of the error of the velocity and the pressure
      on $\mathcal{E}$, with $\nu = 10^{-7}$, plotted against the
      number of globally coupled unknowns.}
  \label{fig:error_vs_dof}
  \end{center}
\end{figure}

From \Cref{prop:local_mass_conservation} and \Cref{rem:edg-div-free} we
know that the error in the divergence of the approximate velocity is
of machine precision for all methods, even on deforming domains.
However, unlike the space-time HDG and EHDG methods, the space-time
EDG method is not divergence-conforming. We, therefore, compute the
$L^2$-norm of the jump of the normal component of the velocity across
facets in the entire space-time domain $\mathcal{E}$ as this is a
measure for the lack of mass conservation. \Cref{tab:edg_jump} shows
optimal rates of convergence when the kinematic viscosity is
$\nu = 10^{-4}$. The rates of convergence for the jump of the normal
component of the velocity across facets is only sub-optimal for the
highly advection-dominated case ($\nu = 10^{-7}$).

\begin{table}[tbp]
  \centering
  \begin{tabular}{cc|cc|cc}
    \hline
    Cells per slab  & Nr. of slabs & $\norm{\jump{u_h \cdot n}}$ & rate 
    & $\norm{\jump{u_h \cdot n}}$ & rate \\ 
    \hline
    \multicolumn{2}{c|}{$\nu = 10^{-4}$} &
    \multicolumn{2}{c|}{$k = 2$} &
    \multicolumn{2}{c}{$k = 3$} \\
    \hline
    $384$   & $20$  & 2.1e-2 &  -  & 1.7e-3 &  -  \\ 
    $1536$  & $40$  & 4.1e-3 & 2.3 & 1.4e-4 & 3.6 \\ 
    $6144$  & $80$  & 7.0e-4 & 2.5 & 9.7e-6 & 3.8 \\ 
    $24576$ & $160$ & 1.0e-4 & 2.8 & 6.6e-7 & 3.9 \\ 
    \hline
    \multicolumn{2}{c|}{$\nu = 10^{-7}$} &
    \multicolumn{2}{c|}{$k = 2$} &
    \multicolumn{2}{c}{$k = 3$} \\
    \hline
    $384$   & $20$  & 2.2e-2 &  -  & 1.9e-3 &  -   \\ 
    $1536$  & $40$  & 4.9e-3 & 2.2 & 1.9e-4 & 3.4   \\ 
    $6144$  & $80$  & 1.0e-3 & 2.3 & 1.8e-5 & 3.4   \\ 
    $24576$ & $160$ & 2.0e-4 & 2.4 & 1.6e-6 & 3.5   \\ 
    \hline
  \end{tabular}
  \caption{Rates of convergence for the jump of the normal velocity
    over facets for the test case describe in \cref{ss:convrates}
    using the space-time EDG discretization.}
  \label{tab:edg_jump}
\end{table}

\subsection{Flow round a rigid cylinder}
\label{ss:rigid_cylinder}

We next consider flow round a cylinder \cite{Lehrenfeld:2016,
  Schaefer:1996}. We solve the Navier--Stokes equations on a fixed
spatial domain $[0,\ 2.2]\times [0,\ 0.41]$ with flow past a
cylindrical obstacle with radius $r = 0.05$ centred at
$(x_1, x_2) = (0.2, 0.2)$. We impose a homogeneous Neumann boundary
condition on the outflow boundary at $x_1=2.2$, while
$u = [6x_2(0.41-x_2)/0.41^2, 0]^T$ is imposed on the inflow boundary
at $x_1 = 0$. On the cylinder and walls $x_2=0$ and $x_2=0.41$ we
impose $u = [0,0]^T$. The initial condition is obtained by solving the
steady Stokes problem. Finally, we set $\nu = 10^{-3}$, $k=3$,
$\Delta t = 5 \cdot 10^{-3}$, $\text{TOL} = 10^{-10}$, and use a
space-time mesh consisting of $9666$ tetrahedra per slab. The velocity
magnitude at final time $t=5$ is shown in \cref{fig:cyl_final}.

\begin{figure}[tbp]
  \begin{center}
    \includegraphics[width=\linewidth]{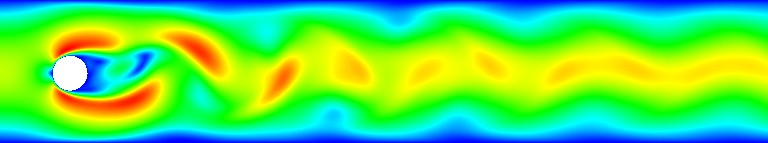}
    \caption{The velocity magnitude of flow around a cylinder, as
      described in \cref{ss:rigid_cylinder}, at $t = 5$ using $3222$
      triangles and $k = 3$ and using the space-time EDG method.}
    \label{fig:cyl_final}
  \end{center}
\end{figure}

Let $\Gamma_c$ denote the space-time boundary of the cylinder. We
define the lift and drag coefficients as
\begin{equation}
  \label{eq:clcd}
  C_L = \frac{1}{r \Delta t} \int_{\Gamma_c} (\sigma n)\cdot e_1,
  \qquad 
  C_D = \frac{1}{r \Delta t} \int_{\Gamma_c} (\sigma n)\cdot e_2,
\end{equation}
where $\sigma = p \mathbb{I} - \nu \nabla u$, $e_1$ and $e_2$ are the
unit vectors in the $x_1$ and $x_2$ directions,
respectively. \Cref{table:cl_cd} contains the minimum and maximum
$C_L$ and $C_D$ values computed during the simulations. These values
compare well to results found in literature \cite{Lehrenfeld:2016,
  Schaefer:1996}.

\Cref{table:cl_cd} also contains the total number globally coupled
degrees-of-freedom for each method. It is clear that even though the
space-time EHDG and EDG have less degrees-of-freedom than the
space-time HDG method, their output is similar.

\begin{table}[tbp]
  \centering
  \begin{tabular}{cccccc}
    \hline
    Method & $\min C_L$ & $\max C_L$ & $\min C_D$ & $\max C_D$ & Nr. of unknowns \\ 
    \hline
   ST-HDG   & -1.014 & 0.98  & 3.153 & 3.219 & 489840 \\
   ST-EHDG & -1.018 & 0.975 & 3.153 & 3.219 & 268944 \\
   ST-EDG   & -1.015 & 0.978 & 3.155 & 3.221 & 158496 \\
    \hline
  \end{tabular}
  \caption{Comparison of the minimum and the maximum lift and drag
    coefficients computed using different space-time methods for flow
    around a rigid cylinder. See \cref{ss:rigid_cylinder}.}
  \label{table:cl_cd}
\end{table}

\subsection{Flow round a forced oscillating cylinder}
\label{ss:forced_cylinder}

In this test case we consider flow round a forced oscillating cylinder
\cite{Calderer:2010}. We solve the Navier--Stokes equations on a
spatial domain $[-6,\ 20]\times [-6,\ 6]$ with flow past a cylindrical
obstacle with radius $r = 0.5$ centred initially at
$(x_1,x_2) = (0, 0)$. We prescribe a vertical oscillatory movement of
the centre of the cylinder for $t \ge 0$ by
\begin{equation*}
  x_2(t) = 0.48\sin(2\pi t/5.94).
\end{equation*}
We apply a homogeneous Neumann boundary condition on the outflow
boundary at $x_1=20$. On the wall boundaries
$x_1 = -6$, $x_2 = 6$, and $x_2 = -6$ we impose $u = [1, 0]^T$, while
$u = [0,0]^T$ is imposed on the cylinder. The initial condition is
obtained by solving the steady Stokes problem and the remaining
parameters are chosen as: $\nu = 10^{-2}$, $\Delta t = 0.025$,
$\text{TOL} = 10^{-9}$, $k=3$, and the computational domain consists
of 20052 tetrahedra per space-time slab.

To accommodate the time-dependent movement of the cylinder, the mesh
is updated at each time step as follows. Nodes inside the spatial box
$\Omega^{\text{in}}(t)=[-2,\ 2] \times [-2+x_2(t),\ 2+x_2(t)]$ move
with the cylinder while nodes outside the spatial box
$\Omega^{\text{out}}=[-4,\ 4] \times [-4,\ 4]$ remain fixed. The
movement of the remaining nodes in
$\Omega^{\text{out}}\backslash\Omega^{\text{in}}(t)$ decreases
linearly with distance. A plot of the mesh when the cylinder is in its
highest and lowest position is given in \cref{fig:forced_cyl_meshes}.
\begin{figure}[tbp]
  \begin{center}
    \subfloat[Cylinder at its highest position.]{\includegraphics[width=.45\linewidth]{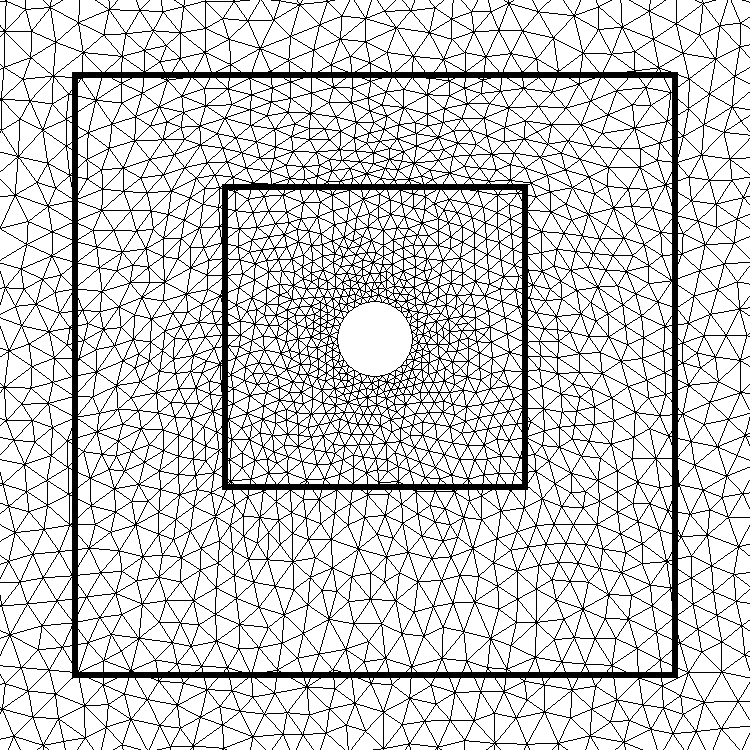}}
    \quad
    \subfloat[Cylinder at its lowest position.]{\includegraphics[width=.45\linewidth]{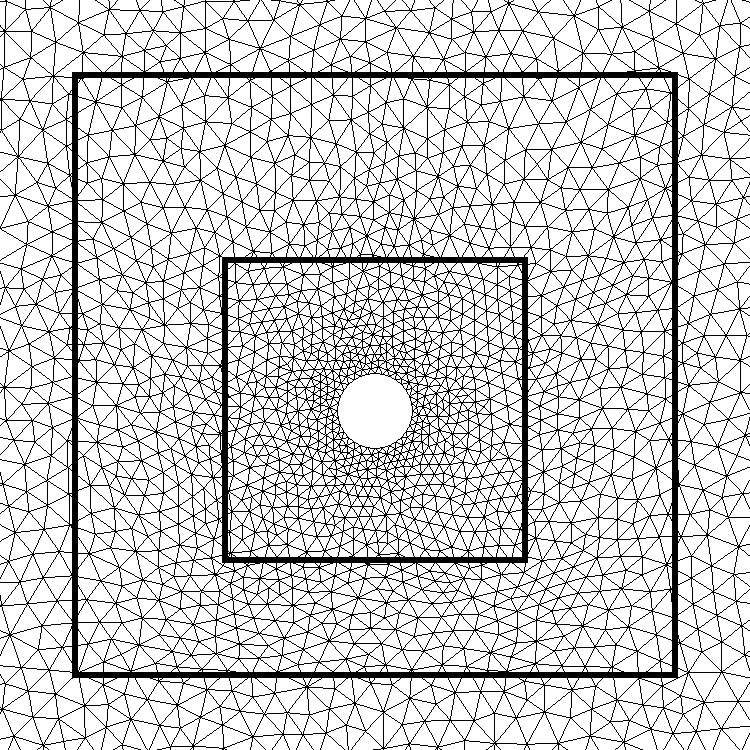}}
    \caption{Mesh deformation during one cycle of the cylinder
      motion. The small square depicts $\Omega^{\text{in}}(t)$ which
      moves with the cylinder. The large square depicts
      $\Omega^{\text{out}}$. See \cref{ss:forced_cylinder}.}
    \label{fig:forced_cyl_meshes}
  \end{center}
\end{figure}

We plot the lift and drag coefficients as a function of position in
\cref{fig:forced_cyl_CLCD}. We observe periodic behaviour in the lift
coefficient, and close to periodic behaviour in the drag
coefficient. There is little difference in the solution computed using
the space-time EDG and EHDG methods.

\begin{figure}[tbp]
  \begin{center}
    \includegraphics[width=.45\linewidth]{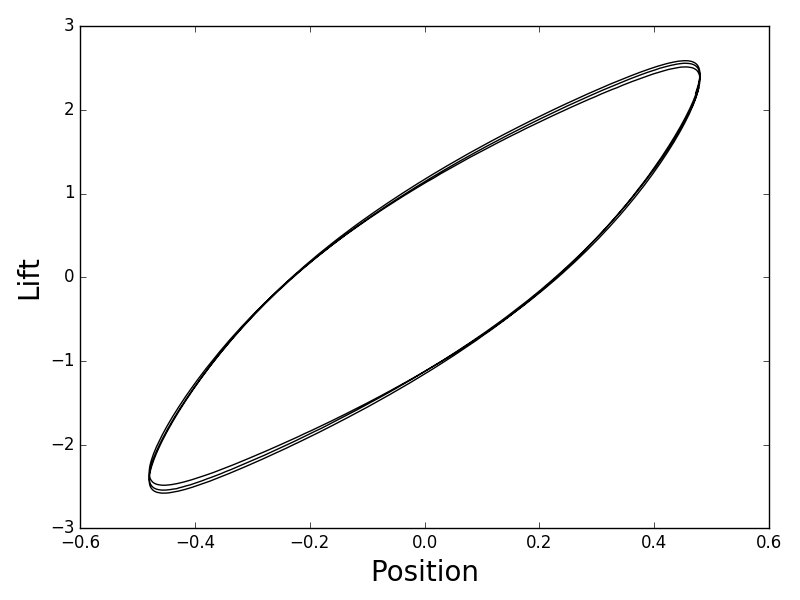}
    \includegraphics[width=.45\linewidth]{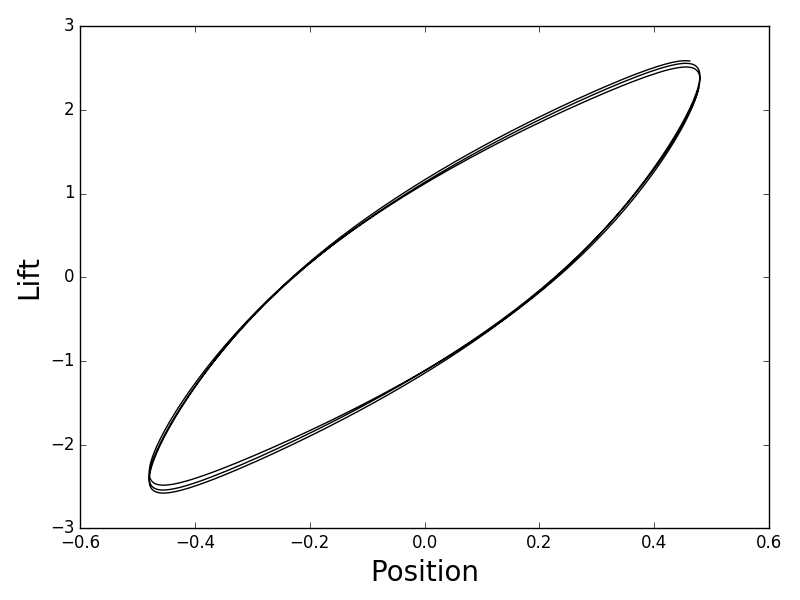}\\
    \includegraphics[width=.45\linewidth]{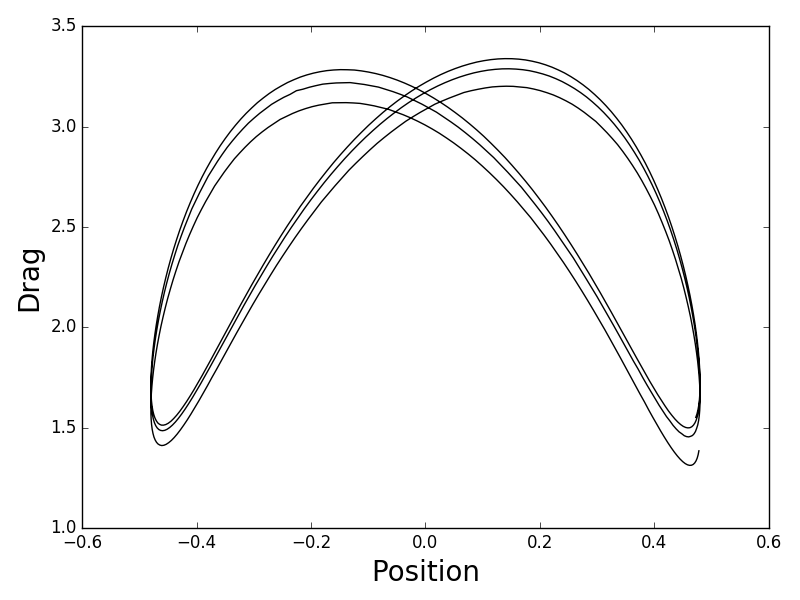}
    \includegraphics[width=.45\linewidth]{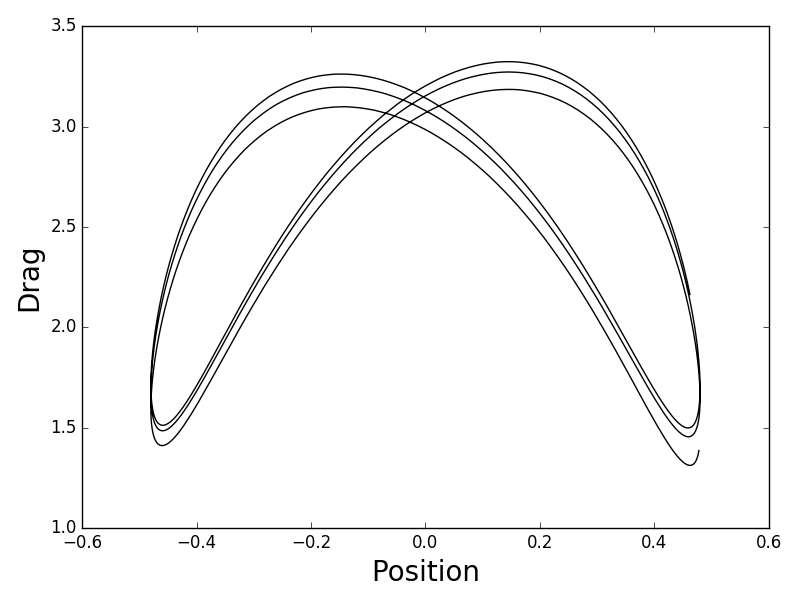}
    \caption{Lift and drag coefficients as a function of the centre of
      the cylinder, for the test case described in
      \cref{ss:forced_cylinder}. Left: the space-time EDG
      method. Right: the space-time EHDG method.}
    \label{fig:forced_cyl_CLCD}
  \end{center}
\end{figure}

The velocity magnitude within one cycle of the cylinder motion
computed using the space-time EDG method is shown in
\cref{fig:forced}. We observe that vortices flow downstreem and that
the flow field around the cylinder at the end of the cycle is similar
to the flow field at the beginning of the cycle. This was observed
also in \cite{Calderer:2010}.

\begin{figure}[tbp]
  \begin{center}
    \subfloat[Solution at $t=17.8$.]{\includegraphics[width=.3\linewidth]{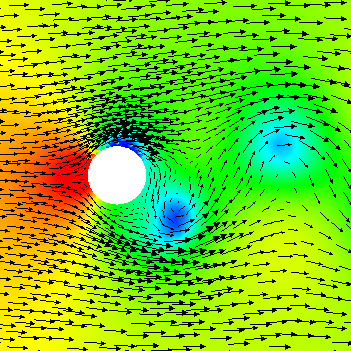}}
    \
    \subfloat[Solution at $t=19.3$.]{\includegraphics[width=.3\linewidth]{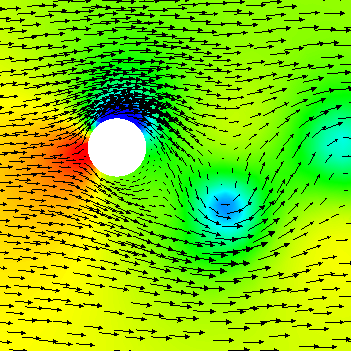}}
    \
    \subfloat[Solution at $t=20.8$.]{\includegraphics[width=.3\linewidth]{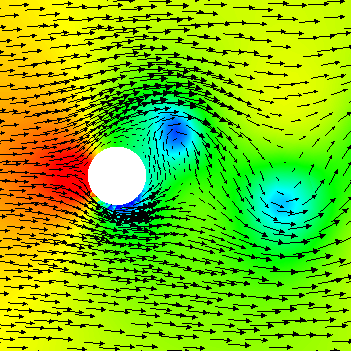}}
    \\
    \subfloat[Solution at $t=22.3$.]{\includegraphics[width=.3\linewidth]{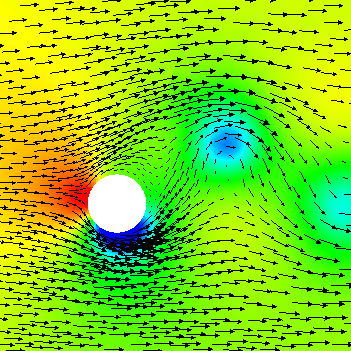}}
    \
    \subfloat[Solution at $t=23.8$.]{\includegraphics[width=.3\linewidth]{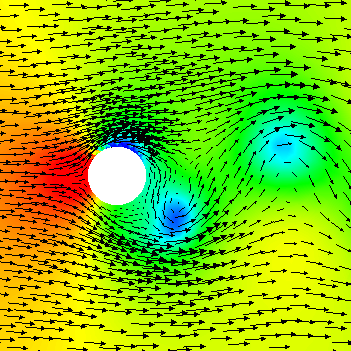}}
    \caption{Velocity vector plot and pressure field around an
      oscillating cylinder, as described in
      \cref{ss:forced_cylinder}. The plots are at different points in
      time in one cycle of the cylinder motion.}
    \label{fig:forced}
  \end{center}
\end{figure}

\subsection{Flow past a pitching and plunging NACA0012 airfoil}
\label{ss:naca0012}

In this final test case we simulate flow around a pitching and
plunging NACA0012 airfoil \cite{Johnson:1994} set in a spatial domain
$[-5,\ 10] \times [-5,\ 5]$. The airfoil is initially at $0^\circ$
angle of attack with trailing edge at $(x_1, x_2) = (-0.5,
0.5)$. Throughout the simulation the trailing edge oscillates
vertically between $x_2 = -0.5$ and $x_2 = 0.5$, while the angle of
attack changes between $-10^\circ$ and $10^\circ$. Both of these
movements happen with a non-dimensional frequency of $0.5$.

The computational domain consists of 24306 tetrahedra per space-time
slab. As parameters we set $k=2$, $\Delta t = 0.01$,
$\text{TOL} = 10^{-7}$, and we set the kinematic viscosity to be
$\nu = 10^{-3}$.

To account for the time-dependent movement of the airfoil, the mesh is
updated at each time step as follows. Nodes within a radius of 1.5
from the trailing edge rotate with the airfoil, nodes outside a radius
of 2 from the trailing edge remain fixed. The vertical movement of the
nodes is treated similarly as in \cref{ss:forced_cylinder}; nodes
inside the spatial box
$\Omega^{\text{in}}(t) = [-3,\ 7] \times [-3 + x_2(t),\ 3 + x_2(t)]$
move with the airfoil, with $x_2(t)$ the $x_2$ coordinate of the
trailing edge, nodes outside the spatial box
$\Omega^{\text{out}} = [-4,\ 8] \times [-4.5\, 4.5]$ remain fixed,
while the movement of the remaining nodes in
$\Omega^{\text{out}} \backslash \Omega^{\text{in}}$ decreases linearly
with distance. We plot the mesh at different instances in time in
\cref{fig:naca_meshes}.

In \cref{fig:naca} we plot the pressure and velocity vector fields,
computed using the space-time EHDG method, for one cycle of the
airfoil motion. When the airfoil is at its highest point, small
vortices detach from the airfoil as the airfoil plunges. In
\cref{fig:naca_a} we see three small vortices, about a chord length
above the airfoil, that detached from the airfoil when it was in its
highest position. These small vortices combine into larger vortices
downstream. A similar process occurs when the airfoil is at its lowest
position, resulting in a street of vortices behind the airfoil. These
observations are in agreement with \cite{Johnson:1994}.

\begin{figure}[tbp]
  \begin{center}
    \subfloat[Mesh at start of cycle.]{\includegraphics[width=.45\linewidth]{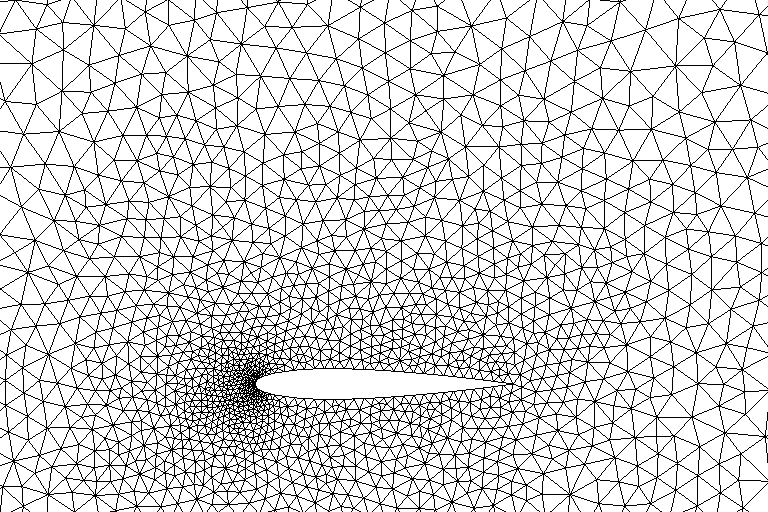}}
    \
    \subfloat[Mesh at quarter of cycle.]{\includegraphics[width=.45\linewidth]{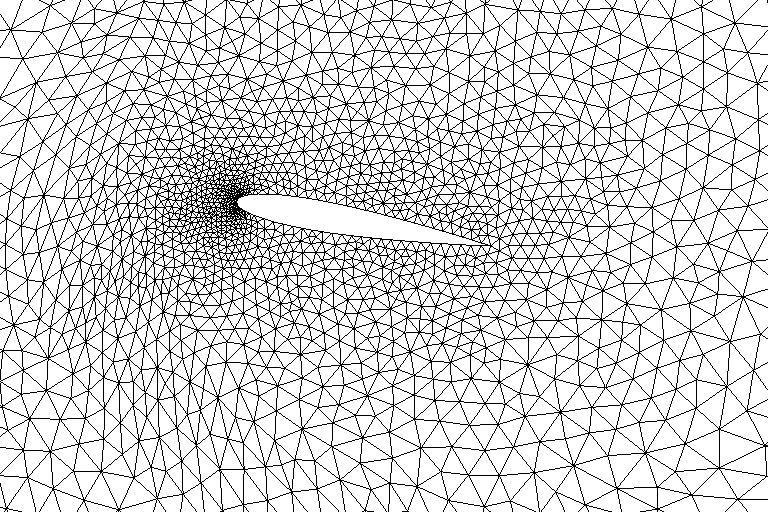}}
    \\
    \subfloat[Mesh halfway through cycle.]{\includegraphics[width=.45\linewidth]{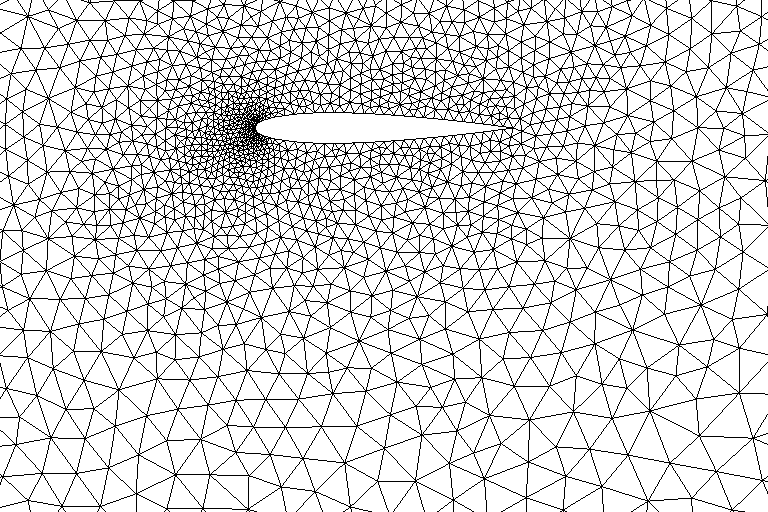}}
    \
    \subfloat[Mesh at three quarters of cycle.]{\includegraphics[width=.45\linewidth]{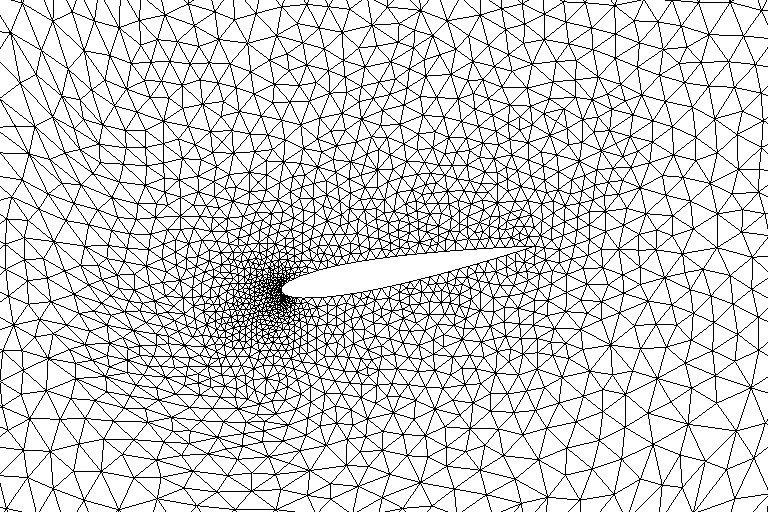}}
    \caption{Mesh deformation of the pitching and plunging NACA0012
      airfoil during one cycle of motion. See \cref{ss:naca0012}.}
    \label{fig:naca_meshes}
  \end{center}
\end{figure}

\begin{figure}[tbp]
  \begin{center}
    \subfloat[Solution at $t=11$. \label{fig:naca_a}]{\includegraphics[width=.45\linewidth]{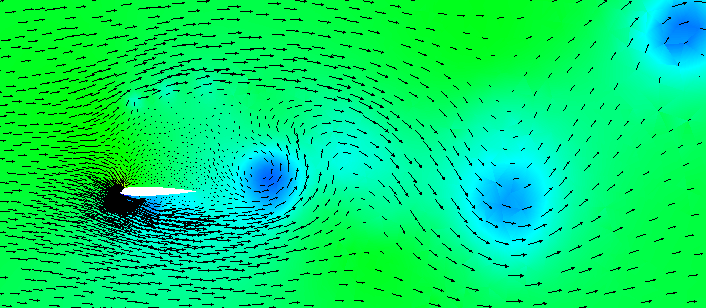}}
    \
    \subfloat[Solution at $t=11.13$.]{\includegraphics[width=.45\linewidth]{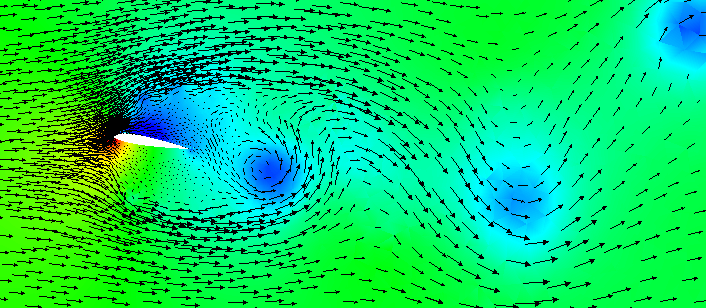}}
    \\
    \subfloat[Solution at $t=11.25$.]{\includegraphics[width=.45\linewidth]{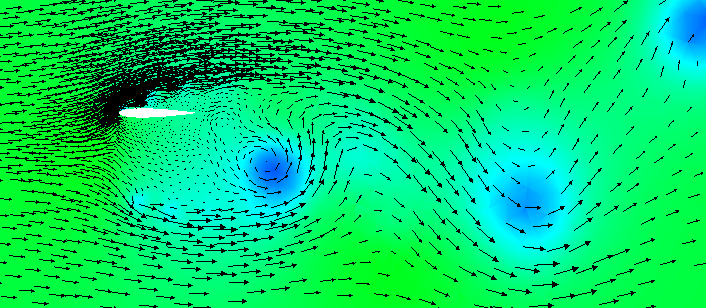}}
    \
    \subfloat[Solution at $t=11.37$.]{\includegraphics[width=.45\linewidth]{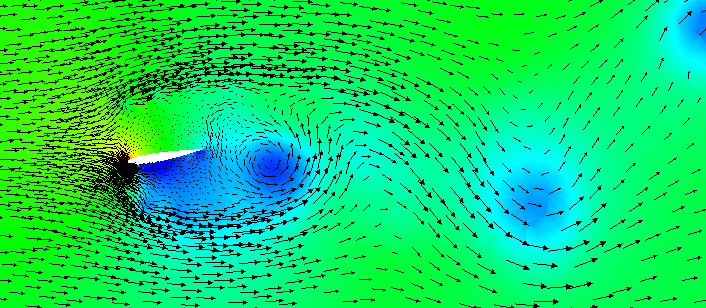}}
    \\
    \subfloat[Solution at $t=11.5$.]{\includegraphics[width=.45\linewidth]{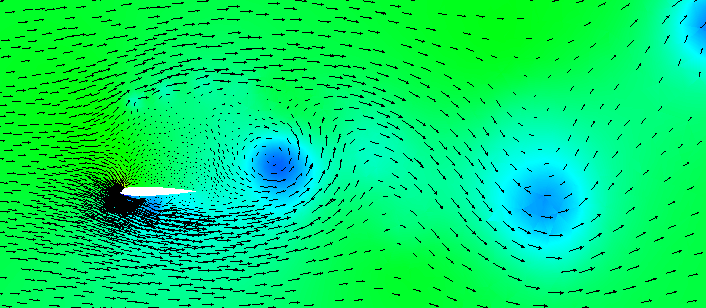}}
    \
    \subfloat[Solution at $t=13$.]{\includegraphics[width=.45\linewidth]{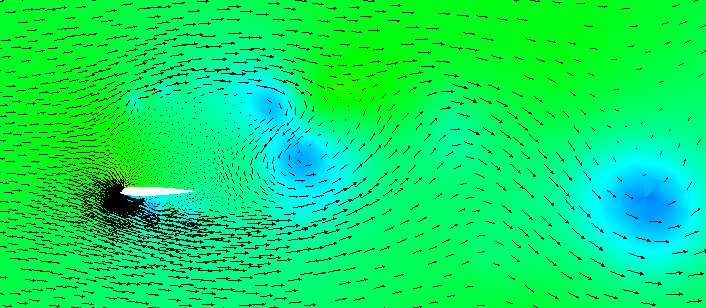}}
    \caption{Velocity vector plot and pressure field around a pitching
      and plunging NACA0012 airfoil, see \cref{ss:naca0012}. From left
      to right and top to bottom the first five pictures show the
      vortices within one cycle of the airfoil motion. The last
      picture shows the vortices three cycles later.}
    \label{fig:naca}
\end{center}
\end{figure}

\section{Conclusions}
\label{s:conclusions}

We presented a space-time embedded-hybridized and a space-time
embedded discontinuous Galerkin finite element method for the
Navier--Stokes equations on moving/deforming domains. Both of these
schemes guarantee a point-wise divergence-free velocity field and are
shown to be energy-stable, even on time-dependent domains. Although
the space-time embedded discontinuous Galerkin method has fewer
globally coupled degrees-of-freedom than the space-time
embedded-hybridized discontinuous Galerkin method, only the latter
discretization conserves mass locally. We have shown the performance
of these methods in terms of rates of convergence, and flow simulations
around a fixed and a moving cylinder and a pitching and plunging
airfoil.

\subsubsection*{Acknowledgements}

SR gratefully acknowledges support from the Natural Sciences and
Engineering Research Council of Canada through the Discovery Grant
program (RGPIN-05606-2015).

\bibliographystyle{elsarticle-num-names}
\bibliography{references}
\end{document}